\documentclass[11pt,parskip=half]{scrartcl}

\usepackage[a4paper,hmargin={2.5cm,2.5cm},vmargin={2.5cm,2.5cm}]{geometry}

\usepackage[utf8]{inputenc}
\usepackage[T1]{fontenc}
\usepackage{lmodern}

\usepackage[leqno]{amsmath}
\usepackage{amssymb,amsthm}
\usepackage{tikz-cd}

\usepackage{bbm}
\usepackage{dsfont}
\usepackage{mathtools}

\theoremstyle{plain}
\newtheorem{theorem}{Theorem}[section]
\newtheorem{lemma}[theorem]{Lemma}
\newtheorem{proposition}[theorem]{Proposition}
\newtheorem{corollary}[theorem]{Corollary}
\theoremstyle{definition}
\newtheorem{definition}[theorem]{Definition}

\newtheorem{example}[theorem]{Example}
\newtheorem{assumption}[theorem]{Assumption}
\theoremstyle{remark}
\newtheorem{remark}[theorem]{Remark}

\usepackage[shortlabels]{enumitem}
\RequirePackage{chngcntr}
\counterwithin*{equation}{section}

\newlist{tfae}{enumerate}{1}%
\setlist[tfae,1]{label=(\roman*)}%

\newcommand{\fv}{\mathfrak{v}}
\newcommand{\fw}{\mathfrak{w}}
\newcommand{\fx}{\mathfrak{x}}
\newcommand{\fy}{\mathfrak{y}}
\DeclareMathOperator{\upc}{\uparrow\!}
\DeclareMathOperator{\downc}{\downarrow\!}
\DeclareMathOperator{\CoAlg}{CoAlg}
\newcommand{\catA}{\catfont{A}}
\newcommand{\catC}{\catfont{C}}
\newcommand{\catI}{\catfont{I}}
\newcommand{\catX}{\catfont{X}}
\newcommand{\SET}{\catfont{Set}}
\newcommand{\ORD}{\catfont{Ord}}
\newcommand{\MET}{\catfont{Met}}
\newcommand{\TOP}{\catfont{Top}}
\newcommand{\PRIEST}{\catfont{Priest}}
\newcommand{\COMPHAUS}{\catfont{CompHaus}}
\newcommand{\STCOMP}{\catfont{StablyComp}}
\newcommand{\POSCH}{\catfont{PosComp}}
\newcommand{\ORDCH}{\catfont{OrdCH}}
\newcommand{\Rels}[1]{#1\text{-}\catfont{Rel}}
\newcommand{\Cats}[1]{#1\text{-}\catfont{Cat}}
\newcommand{\CatCHs}[1]{#1\text{-}\catfont{CatCH}}
\newcommand{\Priests}[1]{{#1\text{-}\catfont{Priest}}}
\newcommand{\op}{\mathrm{op}}
\newcommand{\sep}{\mathrm{sep}}
\newcommand{\ch}{\mathrm{ch}}
\newcommand{\fin}{\mathrm{fin}}
\newcommand{\sym}{\mathrm{sym}}
\newcommand{\two}{\catfont{2}}
\newcommand{\V}{\mathcal{V}}
\newcommand{\catfont}[1]{\mathsf{#1}}
\newcommand{\relto}{\mathrel{\mathmakebox[\widthof{$\xrightarrow{\rule{1.45ex}{0ex}}$}]{\xrightarrow{\rule{1.45ex}{0ex}}\hspace*{-2.4ex}{\mapstochar}\hspace*{1.8ex}}}}
\newcommand{\mate}[1]{\,^\ulcorner\! #1^\urcorner}
\newcommand{\ftD}{\functorfont{D}}
\newcommand{\ftF}{\functorfont{F}}
\newcommand{\ftG}{\functorfont{G}}
\newcommand{\ftH}{\functorfont{H}}
\newcommand{\ftI}{\functorfont{I}}
\newcommand{\ftL}{\functorfont{L}}
\newcommand{\ftP}{\functorfont{P}}
\newcommand{\ftR}{\functorfont{R}}
\newcommand{\ftS}{\functorfont{S}}
\newcommand{\ftT}{\functorfont{T}}
\newcommand{\ftU}{\functorfont{U}}
\newcommand{\ftV}{\functorfont{V}}
\newcommand{\ftId}{\functorfont{Id}}
\newcommand{\ftII}[1]{\catfont{|}#1\catfont{|}}
\newcommand{\ftbF}{\functorfont{\bar{F}}}
\newcommand{\ftbU}{\functorfont{\bar{U}}}
\newcommand{\ftbT}{\functorfont{\bar{T}}}
\newcommand{\Up}{\functorfont{Up}}
\newcommand{\functorfont}{\mathsf}
\newcommand{\mT}{\monadfont{T}}
\newcommand{\mH}{\monadfont{H}}
\newcommand{\mU}{\monadfont{U}}
\newcommand{\monadfont}[1]{\mathbbm{#1}}
\DeclareMathAlphabet{\mathpzc}{OT1}{pzc}{m}{it}
\DeclareMathOperator{\coyoneda}{\mathpzc{h}}
\DeclareMathOperator{\coyonmult}{\mathpzc{w}}
\newcommand{\monad}{(\ftT,m,e)}
\newcommand{\umonad}{(\ftU,m,e)}
\newcommand{\hmonad}{(\ftH,\coyonmult,\coyoneda)}
\newcommand{\N}{\field{N}}
\newcommand{\field}[1]{\mathds{#1}}
\newcommand{\df}[1]{\emph{\textbf{#1}}}

\newcommand\verticalprec{\mathrel{\rotatebox[origin=c]{90}{$\prec$}}}
\DeclareMathOperator{\eupc}{\verticalprec\!}

\newcommand\twoheaddownarrow{\mathrel{\rotatebox[origin=c]{-90}{$\twoheadrightarrow$}}}
\DeclareMathOperator{\tadown}{\twoheaddownarrow}
\DeclareMathOperator{\tbdown}{\Downarrow}

\DeclareFontFamily{U}{mathb}{\hyphenchar\font45}
\DeclareFontShape{U}{mathb}{m}{n}{
<-6> mathb5 <6-7> mathb6 <7-8> mathb7
<8-9> mathb8 <9-10> mathb9
<10-12> mathb10 <12-> mathb12
}{}
\DeclareSymbolFont{mathb}{U}{mathb}{m}{n}
\DeclareMathSymbol{\llcurly}{\mathrel}{mathb}{"CE}
\DeclareMathSymbol{\ggcurly}{\mathrel}{mathb}{"CF}
\newcommand{\ta}{\ggcurly} % to be improved

\DeclareMathOperator{\OLB}{L}
\DeclareMathOperator{\ORB}{R}

\newcommand{\monadb}{(\ftbT,m,e)}
\newcommand{\mbT}{\monadfont{\bar{T}}}

\usepackage[backref=page,hypertexnames=false,colorlinks=true]{hyperref}

\begin{document}

\title{Hausdorff coalgebras}

\author{Dirk Hofmann and Pedro Nora%
  \thanks{Research partially supported by Funda\c{c}\~ao para a Ci\^encia e a
    Tecnologia (FCT), within projects UID/MAT/04106/2019 (CIDMA),
    POCI-01-0145-FEDER-030947 and by the ERDF – European Regional Development
    Fund through the Operational Programme for Competitiveness and
    Internationalisation - COMPETE 2020 Programme.}}

\publishers{\small{Center for Research and Development in Mathematics and
    Applications, Department of Mathematics, University of Aveiro, Portugal.

    HASLab INESC TEC - Institute for Systems and Computer Engineering,
    Technology and Science, University of Minho, Portugal.

  \texttt{dirk@ua.pt},\quad \texttt{pedro.top.nora@gmail.com}}}

\date{}

\maketitle

\begin{abstract}
  As composites of constant, (co)product, identity, and powerset functors,
  Kripke polynomial functors form a relevant class of $\SET$-functors in the
  theory of coalgebras. The main goal of this paper is to expand the theory of
  limits in categories of coalgebras of Kripke polynomial functors to the
  context of quantale-enriched categories.  To assume the role of the powerset
  functor we consider ``powerset-like'' functors based on the Hausdorff
  $\V$-category structure. As a starting point, we show that for a lifting of a
  $\SET$-functor to a topological category $\catX$ over $\SET$ that commutes
  with the forgetful functor, the corresponding category of coalgebras over
  $\catX$ is topological over the category of coalgebras over $\SET$ and,
  therefore, it is ``as complete'' but cannot be ``more complete''. Secondly,
  based on a Cantor-like argument, we observe that Hausdorff functors on
  categories of quantale-enriched categories do not admit a terminal
  coalgebra. Finally, in order to overcome these ``negative'' results, we
  combine quantale-enriched categories and topology \emph{à la} Nachbin. Besides
  studying some basic properties of these categories, we investigate
  ``powerset-like'' functors which simultaneously encode the classical Hausdorff
  metric and Vietoris topology and show that the corresponding categories of
  coalgebras of ``Kripke polynomial'' functors are (co)complete.
\end{abstract}

\tableofcontents

\section{Introduction}

Starting with early studies in the nineties to the introduction of uniform
notions of \emph{behavioural metric} in the last decade the study of coalgebras
over metric-like spaces has focused on four specific areas:
\begin{enumerate}
\item liftings of functors from the category $\SET$ of sets and functions to
  categories of metric spaces (see \cite{BBKK18,BK16,BKV19}), as a way of
  lifting state-based transition systems into transitions systems over
  categories of metric spaces;
\item results on the existence of terminal coalgebras and their computation (see
  \cite{TR98,BBKK18}), as a way of calculating the behavioural distance of two
  given states of a transition system;
\item the introduction of behavioural metrics with corresponding ``Up-To
  techniques'' (see \cite{BKP18,BBKK18,vBHMW05}), as a way of easing the
  calculation of behavioural distances;
\item and the development of coalgebraic logical foundations over metric spaces
  (see \cite{BK16,KM18,WSPK18}), as a way of reasoning in a quantitive way about
  transition systems.
\end{enumerate}

In this paper we focus on the first two topics, with particular interest in
metric versions of \emph{Kripke polynomial functors}. As composites of constant,
(co)product, identity, and powerset functors, Kripke polynomial functors form a
pertinent class of $\SET$-functors in the theory of coalgebras (for example,
see~\cite{Rut00a}, \cite{BRS09} and \cite{KKV04}), which is well-behaved in
regard to the existence of limits in their respective categories of coalgebras
--- assuming that the powerset functor is submitted to certain cardinality
restrictions. The latter constraint is essential since the powerset functor
$\ftP\colon\SET\to\SET$ does not admit a terminal coalgebra; a well-known fact
which follows from the following:
\begin{itemize}
\item in \cite{Lam68a} it is shown that the terminal coalgebra of a functor
  $\ftF \colon\catC\to\catC$ is a fixed point of $\ftF$, and
\item in \cite{Can91} it is (essentially) proven that the powerset functor
  $\ftP \colon \SET\to\SET$ does not have fixed points.
\end{itemize}
On the other hand, being accessible, the finite powerset functor
$\ftP_\fin\colon\SET\to\SET$ does admit a terminal coalgebra (see \cite{Bar93});
in fact, the category of coalgebras for $\ftP_\fin\colon\SET\to\SET$ is
complete. Metric counterparts of the powerset functor are often based on the
Hausdorff metric, informally, we call them \emph{Hausdorff functors}. This
metric was originally introduced in \cite{Hau14,Pom05} (see also \cite{BT06}),
and, recently, has been considered in the more general context of quantale
enriched categories (see \cite{ACT10,Stu10}) in which we discuss the results
presented here.

A common theme of the papers \cite{BBKK18} and \cite{BKV19} mentioned in the
first point above is that the authors study liftings of $\SET$-functors to
categories of metric spaces, or more generally to the category $\Cats{\V}$ of
$\V$-categories and $\V$-functors, in the sense that the diagram

\begin{center}
  \begin{tikzcd}[ampersand replacement=\&]
    \Cats{\V} \ar[r,"\ftbT"]\ar[d,""'] \& \Cats{\V} \ar[d,""] \\
    \SET \ar[r,"\ftT"'] \& \SET
  \end{tikzcd}
\end{center}

commutes.  In Section~\ref{sec:gener-cons} we show that, for such a lifting of a
$\SET$-functor, the corresponding category of coalgebras over $\Cats{\V}$ is
topological over the category of coalgebras over $\SET$ (see Theorem~\ref{p:3}).
This implies that it is possible to recast over $\Cats{\V}$ all the theory about
limits in categories of Kripke polynomial coalgebras over $\SET$. However, this
result also highlights that ``adding a $\V$-category structure'' does not
improve the situation regarding limits by itself. In particular, the Hausdorff
functor that considers all subsets of a metric space does not admit a terminal
coalgebra.

Besides cardinal restrictions, another way to ``tame'' the powerset functor is
to equip a set with some kind of structure and then consider only its
``structure relevant'' subsets. This is precisely the strategy employed in
\cite{HNN19} where we passed from Kripke polynomial functors to \emph{Vietoris
  polynomial functors} on categories of topological spaces. For instance, it is
implicitly shown in \cite[page 245]{Eng89} that the classic Vietoris functor on
the category of compact Hausdorff spaces and continuous maps has a terminal
coalgebra, and this result generalises to all topological spaces when
considering the compact Vietoris functor on $\TOP$ which sends a space to its
hyperspace of compact subsets (see \cite{HNN19} for details). This fact might
not come as a surprise for the reader thinking of compactness as ``generalised
finiteness''; however, it came as a surprise to us to learn that the lower
Vietoris functor on $\TOP$, where one considers all closed subsets, also admits
a terminal coalgebra.

Motivated by the fact that finite topological spaces correspond precisely to
finite ordered sets, over the past decades several results about topological
spaces have been inspired by their finite counterparts; for a sequence of
results see for instance \cite{JS02,JS02a,CH02}. One therefore might wonder if
the result regarding the lower Vietoris functor on $\TOP$ has an order-theoretic
counterpart; in other words, does the upset functor $\Up \colon\ORD\to\ORD$
admit a terminal coalgebra? The answer is negative, as it follows from the
``generalized Cantor Theorem'' of \cite{DG62}.%
\footnote{We thank Adriana Balan for calling our attention to \cite{DG62}.}
Based on \cite{DG62}, in Section~\ref{sec:Hpf} we generalise Cantors Theorem
further (see Theorem~\ref{d:thm:1}) and use this result to show that the
(non-symmetric) Hausdorff functor on $\Cats{\V}$ -- sending a metric space to
the space of all ``up-closed'' subsets -- does not admit a terminal coalgebra.

To overcome these ``negative results'' regarding completeness of categories of
coalgebras, in Section~\ref{sec:hausd-polyn-funct} we add a topological
component to the $\V$-categorical setting. More specifically, we introduce the
Hausdorff construction for $\V$-categories equipped with a compatible compact
Hausdorff topology. We note that these \emph{$\V$-categorical compact Hausdorff
  spaces} are already studied in \cite{Tho09,HR18}, being the corresponding
category denoted here by $\CatCHs{\V}$. Also, we find it worthwhile to notice
that the notion of $\V$-categorical compact Hausdorff space generalises
simultaneously Nachbin's ordered compact Hausdorff spaces \cite{Nac50} and the
classic notion of compact metric space; therefore, it provides a framework to
combine and even unify both theories. For example:
\begin{itemize}
\item It is known that the specialisation order of a sober space is directed
  complete (see \cite[Lemma~II.1.9]{Joh86}); in \cite{HR18} we observed that
  this fact implies immediately that the order relation of an ordered compact
  Hausdorff space is directed complete. Furthermore, an appropriate version of
  this result in the quantale-enriched setting implies that the metric of a
  metric compact Hausdorff space (i.e.\ a metric space with a \emph{compatible}
  compact Hausdorff topology) is Cauchy complete, generalising the classical
  fact that a compact metric space (i.e.\ a metric space where the
  \emph{induced} topology is compact) is Cauchy complete.
\item The Hausdorff functor $\ftH \colon \CatCHs{\V}\to \CatCHs{\V}$ introduced
  in Section~\ref{sec:hausd-polyn-funct} combines the Vietoris topology and the
  Hausdorff metric; in particular, for a metric compact Hausdorff space, the
  Hausdorff metric is compatible with the Vietoris topology
  (Proposition~\ref{d:prop:6}). This result represents a variation of the
  classic fact stating that, for every compact metric space $X$, the Hausdorff
  metric induces the Vietoris topology of the compact Hausdorff space $X$ (see
  \cite{Mic51}).
\end{itemize}
By ``adding topology'', and under some assumptions on the quantale $\V$, we are
able to show that $\ftH \colon \CatCHs{\V}\to \CatCHs{\V}$ preserves codirected
limits (see Theorem~\ref{d:thm:5}); which eventually allows us to conclude that,
for every Hausdorff polynomial functor on $\CatCHs{\V}$, the corresponding
category of coalgebras is complete (see Theorem~\ref{d:thm:6}).

In the last part of this paper we consider a $\V$-categorical counterpart of the
notion of a Priestley space. In \cite{HN18} we developed already ``Stone-type''
duality theory for these type of spaces; here we show that
$\ftH \colon \CatCHs{\V}\to \CatCHs{\V}$ sends Priestley spaces to Priestley
spaces, generalising a well-known fact of the Vietoris functor on the category
of partially ordered compact spaces. Consequently, many results regarding
coalgebras for $\ftH \colon \CatCHs{\V}\to \CatCHs{\V}$ are valid for its
restriction to Priestley spaces as well.

\textbf{Acknowledgement}. We are grateful to Renato Neves for many fruitful
discussions on the topic of the paper, without his input this work would not
exist.

\section{Strict functorial liftings}
\label{sec:gener-cons}

The main motif of this work is to expand the study of limits in categories of
coalgebras of Kripke polynomial functors to the context of quantale-enriched
categories. In more general terms, this means that given an endofunctor $\ftF$
on a category $\catA$ and a faithful functor $\ftU \colon \catX \to \catA$, our
problem consists in studying a ``lifting'' of $\ftF$ to an endofunctor $\ftbF$
on $\catX$.  In a strict sense, by ``lifting'' we mean that the diagram

\begin{equation}
  \label{p:28}
  \begin{tikzcd}[row sep=large, column sep=large,ampersand replacement=\&]
    \catX \ar[r,"\ftbF"] \ar[d,"\ftU"'] \& \catX \ar[d,"\ftU"] \\
    \catA \ar[r,"\ftF"'] \& \catA
  \end{tikzcd}
\end{equation}

commutes.

\begin{remark}
  If in \eqref{p:28} the functor $\ftbF$ has a fix-point, then so has
  $\ftF$. Hence, if $\ftF$ does not have a fix-point, then neither does
  $\ftbF$. In particular, any strict lifting of the powerset functor
  $\ftP \colon\SET\to\SET$ does not admit a terminal coalgebra.
\end{remark}

Then, we obtain a faithful functor

\begin{equation*}
  \ftbU \colon \CoAlg(\ftbF) \to \CoAlg(\ftF)
\end{equation*}

by ``applying $\ftU$''.  In \cite[Theorem~3.11]{HNN19} we showed under
additional conditions that, if the forgetful functor
$\ftU \colon \catX \to \catA$ is topological, then so is the functor
$\ftbU \colon \CoAlg(\ftbF) \to \CoAlg(\ftF)$. We start by improving upon this
result.

In the remainder of this section, let $\ftU \colon \catX \to \catA$ be a
topological functor, for more information we refer to \cite{AHS90}. We recall
that $\catX$ is fibre-complete, and for an object $A$ of $\catA$ we use the
suggestive notation $(A,\alpha)$ to denote an element of the fiber of $A$. Then
we write $\alpha \leq \beta$ if $ 1_A \colon (A, \alpha) \to (A,\beta)$ is a
morphism of $\catX$. Since we also assume the existence of functors
$\ftF \colon \catA \to \catA$ and $\ftbF \colon \catX \to \catX$ such that the
diagram~\eqref{p:28} commutes, with a slight abuse of notation, we often write
$(\ftF A,\ftF\alpha)$ instead of $\ftbF(A,\alpha)$.

For a $\ftU$-structured arrow $f \colon A \to \ftU(B, \beta)$ in $\catA$, we
denote by $(A, f_{\beta}^\triangleleft)$ the corresponding $\ftU$-initial lift.
Similarly, for $f \colon \ftU(A,\alpha) \to B$ in $\catA$, we denote by
$(B, f_{\alpha}^\triangleright)$ the corresponding $\ftU$-final lift. Below we
collect some well-known facts.

\begin{proposition}
  Let $f \colon A \to B$ be a morphism in $\catA$ and $(A, \alpha)$ and
  $(B, \beta)$ be objects in the fibres of $A$ and $B$, respectively. Then the
  following assertions are equivalent.
  \begin{tfae}
  \item $f \colon (A, \alpha) \to (B, \beta)$ is a morphism in $\catX$.
  \item $\alpha \leq f_\beta^\triangleleft$.
  \item $f_\alpha^\triangleright \leq \beta$.
  \end{tfae}
\end{proposition}

\begin{proposition}
  \label{p:0}
  Let $(A,\alpha)$ and $(A,\beta)$ be objects in the fibre of an object $A$ of
  $\catA$. If $\alpha \leq \beta$ then $\ftF\alpha \leq \ftF\beta$.
\end{proposition}

\begin{proposition}
  \label{p:1}
  Let $c \colon A \to \ftF A$ be a morphism in $\catA$ and let $\mathcal{A}$ be
  a collection of objects $(A,\alpha)$ in the fibre of $A$ such that
  $c \colon (A, \alpha) \to (\ftF A, \ftF\alpha)$ is in $\catX$. Let
  $(A,\alpha_c)$ be the supremum of $\mathcal{A}$. Then,
  $c \colon (A, \alpha_c) \to (\ftF A, \ftF\alpha_c)$ is a morphism of $\catX$.
\end{proposition}
\begin{proof}
  First note that
  \begin{equation*}
    (A,\alpha_c)
    \xrightarrow{\quad c\quad}\bigvee\{(\ftF A,\ftF\alpha)\mid(X,\alpha)\in\mathcal{A}\}
  \end{equation*}
  is a morphism $\catX$, and, by Proposition~\ref{p:0}, so is
  \begin{equation*}
    \bigvee\{(\ftF A,\ftF\alpha)\mid(X,\alpha)\in\mathcal{A}\}
    \xrightarrow{\quad 1_{\ftF A}\quad} (\ftF A,\ftF \alpha_c).\qedhere
  \end{equation*}
\end{proof}

\begin{theorem}
  \label{p:3}
  The functor $\ftU \colon \CoAlg(\ftbF) \to \CoAlg(\ftF)$ is topological.
\end{theorem}
\begin{proof}
  Let $(A_i,\alpha_i, c_i)_{i \in I}$ be a family of objects in $\CoAlg(\ftbF)$,
  and $(f_i \colon (A, c) \to (A_i, c_i))_{i \in I}$ a cone in $\CoAlg(\ftF)$.
  Consider
  \begin{equation*}
    \alpha_c
    = \bigvee \{ \alpha\mid c \colon (A,\alpha)\to (\ftF A,\ftF \alpha)\text{ is in
      $\catX$ and, for all $i\in I$, }\alpha\le f_{\alpha_i}^\triangleleft\}.
  \end{equation*}
  Then, by Proposition~\ref{p:1},
  $c \colon (A,\alpha_c) \to (\ftF A,\ftF\alpha_c)$ is a morphism of $\catX$.
  Moreover, by construction, $\alpha_c \leq f_{\alpha_i}^\triangleleft$ for all
  $i\in I$; hence, $(f_i \colon (A, \alpha_c) \to (A_i,\alpha_i))_{i \in I}$ is
  a cone in $\catX$.  Therefore,
  $(f_i \colon (A, \alpha_c, c) \to (A_i,\alpha_i,c_i))_{i \in I}$ is a cone in
  $\CoAlg(\ftbF)$. We claim that this cone is $\ftU$-initial.

  Let $(g_i \colon (B,\beta,b) \to (A_i, \alpha_i, c_i)$ be a cone in
  $\CoAlg(\ftbF)$, and $h \colon (B,b) \to (A, c)$ a morphism in $\CoAlg(\ftF)$
  such that, for every $i \in I$,
  \begin{equation}
    \label{p:4}
    f_i \cdot h = g_i
  \end{equation}
  We will see that $h_\beta^\triangleright \leq \alpha_c$. First observe that it
  follows from \eqref{p:4} that
  $h_\beta^\triangleright \leq f_{\alpha_i}^\triangleleft$ for all $i\in
  I$. Furthermore, since $c \cdot h = Fh \cdot b$ in $\catA$ it follows that
  $c \colon (A, h_\beta^\triangleright) \to (FA, F(h_\beta^\triangleright))$ is
  a morphism of $\catX$ because
  $h \colon (B, \beta) \to (A, h_\beta^\triangleright)$ is final. Therefore, by
  construction of $\alpha_c$, we conclude that
  $h_\beta^\triangleright \leq \alpha_c$.
\end{proof}

\begin{corollary}
  \label{p:13}
  The category $\CoAlg(\ftbF)$ has limits of shape $I$ if and only if
  $\CoAlg(\ftF)$ has limits of shape $I$. In particular, $\CoAlg(\ftbF)$ has a
 terminal object if and only if $\CoAlg(\ftF)$ has one.
\end{corollary}

This means that $\CoAlg(\ftbF)$ cannot be ``more complete'' than $\CoAlg(\ftF)$,
one of the reasons why in Section~\ref{sec:hausdorff-functor} we will
lift the powerset functor on $\SET$ to $\Cats{\V}$ but only ``up to natural
transformation''.

On the other hand, Corollary~\ref{p:13} also means that $\CoAlg(\ftbF)$ is ``at
least as complete'' as $\CoAlg(\ftF)$, which allow us to recover known results
about the existence of limits in $\CoAlg(\ftbF)$.  For example, in \cite[Theorem
6.2]{BBKK18} it is proven, by implicitly constructing the right adjoint of
$\ftU$, that every lifting to the category of symmetric metric spaces of an
endofunctor on $\SET$ that admits a terminal coalgebra also admits a terminal
coalgebra. A similar result was also obtained in~\cite[Theorem 4.15]{BKV19} for
``$\V$-Catifications'' --- very specific liftings from $\SET$ to $\Cats{\V}$.

Note that Theorem~\ref{p:3} even tell us how to construct limits in
$\CoAlg(\ftbF)$ from limits in $\CoAlg(\ftF)$.  In particular, if $\ftU$ is a
forgetful functor to $\SET$ then a limit in $\CoAlg(\ftbF)$ has the same
underlying set of the corresponding limit in $\CoAlg(\ftF)$.  This behaviour was
already observed in \cite[Theorem 4.16]{BKV19} for some particular liftings to
$\Cats{\V}$.

\begin{example}
  Given a subfunctor $\ftF$ of the powerset functor on $\SET$, the corresponding
  class of Kripke polynomial functors is typically defined as the smallest class
  of $\SET$-functors that contains the identity functor, all constant functors
  and it is closed under composition with $\ftF$, sums and product of
  functors. If we are interested in strict liftings to $\Cats{\V}$, then
  Theorem~\ref{p:3} tells us that is possible to recast over $\Cats{\V}$ all the
  theory about limits in categories of Kripke polynomial coalgebras over $\SET$.
  For example, if we consider a strict lifting of the finite powerset functor,
  then every category of coalgebras of a Kripke polynomial functor is
  (co)complete, and every limit is obtained as the initial lift of the
  corresponding limit of $\SET$-coalgebras.
\end{example}

In the sequel, we give an example of a generic way of lifting a functor $\ftF
\colon \catA \to \catA$ to a category $\catX$ that is topological over $\catA$.
In particular, this construction is used to lift  $\SET$-functors to categories
of metric spaces in \cite{BBKK18}, and to categories of $\V$-categories in
\cite{BKV19}.

For a functor $\ftF \colon \catA \to \catA$ and $\catA$-morphisms
$\psi \colon A \to \widetilde{A}$ and
$\sigma \colon\ftF \widetilde{A} \to \widetilde{A}$, we denote by
$\psi^\Diamond \colon \ftF A \to \widetilde{A}$ the composite

\begin{equation*}
  \ftF A \xrightarrow{\quad \ftF\psi \quad} \ftF\widetilde{A} \xrightarrow{\quad \sigma \quad} \widetilde{A}
\end{equation*}

in $\catA$.

Consider now a category $\catX$ equipped with a topological functor
$\ftII{-} \colon\catX\to\catA$ and an $\catX$-object $\widetilde{X}$ whose
underlying set $\ftII{\widetilde{X}}$ carries the structure
$\sigma \colon \ftF\ftII{\widetilde{X}}\to\ftII{\widetilde{X}}$ of a
$\ftF$-algebra. Then
$(\psi^\Diamond \colon \ftF
\ftII{X}\to\ftII{\widetilde{X}})_{\psi\in\catX(X,\widetilde{X})}$ is a
$\ftII{-}$-structured cone, and we define $\ftbF X$ to be the domain of the
initial lift of this cone. Clearly:

\begin{theorem}\label{d:thm:2}
  \begin{enumerate}
  \item The construction above defines a functor $\ftbF \colon\catX\to\catX$
    making the diagram
    \begin{equation*}
      \begin{tikzcd}[row sep=large, column sep=large]
        \catX \ar[r,"\ftbF"]\ar[d,"\ftII{-}"'] & \catX
        \ar[d,"\ftII{-}"] \\
        \catA \ar[r,"\ftF"'] & \catA
      \end{tikzcd}
    \end{equation*}
    commutative.
  \item For every $\psi \colon X\to\widetilde{X}$ in $\catX$, $\psi^\Diamond$ is
    an $\catX$-morphism $\psi^\Diamond \colon \ftbF X\to\widetilde{X}$. In
    particular, $\sigma={1_{\widetilde{X}}}^\Diamond$ is an $\catX$-morphism
    $\sigma \colon \ftbF\widetilde{X}\to \widetilde{X}$.
  \item\label{d:item:1} If $\widetilde{X}$ is injective with respect to initial
    morphisms, then $\ftbF \colon\catX\to\catX$ preserves initial morphism
    (compare with \cite[Theorem~5.8]{BBKK18}).
  \item \label{d:item:2} Let $\alpha \colon \ftF\Rightarrow \ftG$ be a natural
    transformation such that
    $\sigma_{\ftG}\cdot \alpha_{\widetilde{X}}=\sigma_{\ftF}$.  Then $\alpha$
    lifts to a natural transformation between the corresponding
    $\catX$-functors.
  \item If $\ftF=\ftT$ is part of a monad $\mT=\monad$ on $\catA$ and
    $\sigma\colon \ftT\ftII{\widetilde{X}}\to\ftII{\widetilde{X}}$ is a
    $\mT$-algebra, then $\mT$ lifts naturally to a monad $\mbT=\monadb$ on
    $\catX$.
  \end{enumerate}
\end{theorem}
\begin{proof}
  The first affirmation follows immediately from the commutativity of the
  diagram
  \begin{equation*}\label{d:eq:1}
    \begin{tikzcd}%[row sep=large, column sep=large]
      \ftF X \ar[r,"\ftF f"]\ar[dr,"(\psi\cdot f)^{\Diamond}"'] & \ftF Y \ar[d,"\psi^\Diamond"] \\
      & \widetilde{X},
    \end{tikzcd}
  \end{equation*}
  and similarly the last two ones. The second affirmation is true by
  definition. In Proposition~\ref{d:prop:7} we prove a slightly more general
  version of \eqref{d:item:1}.
\end{proof}

\begin{remark}
  We note that in Theorem~\ref{d:thm:2}~\eqref{d:item:2}, the inequality
  $\sigma_{\ftG} \cdot \alpha_{\widetilde{X}} \le \sigma_{\ftF}$ does not
  guarantee that $\alpha_X \colon \ftF X \to \ftG X$ is an $\catX$-morphism
  (this contradicts \cite[Theorem~8.1]{BBKK18}). For instance, consider
  $\catX = \MET_\sym$, $\widetilde{X}=[0,\infty]$ and
  $\ftF,\ftG \colon \SET\to\SET$ with $\ftF=\ftG$ being the identity functor on
  $\SET$, $\lambda=1$, $\sigma_{\ftG}=1_{[0,\infty]}$ and $\sigma_{\ftF}=\infty$
  (constant). Clearly,
  $\sigma_{\ftG}\cdot\lambda_{[0,\infty]}\le\sigma_{\ftF}$. However,
  $\ftG \colon \MET_\sym\to\MET_\sym$ is the identity functor and
  $\ftF \colon \MET_\sym\to\MET_\sym$ transforms every symmetric metric space
  into the indiscrete space on the same underlying set. Hence, for a
  non-indiscrete space $X$, $\lambda_X \colon \ftF X\to \ftG X$ is not a
  morphism in $\MET_\sym$.
\end{remark}

In this context it is useful to note that Theorem~\ref{d:thm:2}\eqref{d:item:1}
gives a sufficient condition for the preservation of initial morphisms that can
be formulated in a slightly more general way.

\begin{proposition}
  \label{d:prop:7}
  Let $\ftF \colon\catX\to\catX$ be a functor,
  $\sigma \colon \ftF\widetilde{X} \to \widetilde{X}$ a morphism in $\catX$, and
  $\ftII{-} \colon \catX \to \SET$ a faithful functor. Assume further that
  $\widetilde{X}$ is injective in $\catX$ with respect to initial morphisms and,
  for every object $X$ in $\catX$, the cone
  $(\psi^\Diamond \colon \ftF
  X\to\widetilde{X})_{\psi\in\catX(X,\widetilde{X})}$ in $\catX$ is
  initial. Then $\ftF$ preserves initial morphisms.
\end{proposition}
\begin{proof}
  Let $f \colon X \to Y$ be an initial morphism in $\catX$.  Since
  $\widetilde{X}$ is injective with respect to initial morphisms, every morphism
  $\psi \colon X \to \widetilde{X}$ in $\catX$ factors as $h_\psi \cdot f$, for
  some $h_\psi \colon Y \to \widetilde{X}$ in $\catX$.  Hence, $\ftF\psi = \ftF
  h_\psi \cdot \ftF f$.  Now, suppose that $h \colon Z \to \ftF Y$ is a morphism
  in $\catX$, and $g \colon \ftII{Z} \to \ftII{\ftF X}$ is a function such that
  $\ftII{\ftF f} \cdot g = \ftII{h}$. Then, for every morphism $\psi \colon X
  \to \widetilde{X}$ in $\catX$, we have
  \[
    \ftII{\psi^\Diamond} \cdot g = \ftII{\sigma \cdot h_\psi \cdot \ftF f} \cdot g
                                 = \ftII{\sigma \cdot h_\psi \cdot h}.
  \]
  Therefore, the claim follows because the cone $(\psi^\Diamond \colon \ftF X
  \to \widetilde{X})_{\psi\in\catX(X,\widetilde{X})}$ is initial and $\ftII{-}$
  is faithful.
\end{proof}

The injectivity-condition on $\widetilde{X}$ is often fulfilled; the proposition
below collects some examples.

\begin{proposition}
  \begin{enumerate}
  \item The $\V$-category $(\V, \hom)$ is injective in $\Cats{\V}$ with respect
    to initial morphisms.  Since $\Cats{\V}_\sym\hookrightarrow\Cats{\V}$
    preserves initial morphisms (see Theorem~\ref{p:22}), the symmetrisation
    of $(\V, \hom)$ is injective in $\Cats{\V}_\sym$.
  \item The unit interval $[0,1]$ is injective in $\POSCH$ with respect to
    initial morphisms (see \cite{Nac50}).
  \item The Sierpiński space is injective with respect to initial morphisms in
    the category $\TOP$ of topological spaces and continuous maps.
  \end{enumerate}
\end{proposition}
The next proposition shows that the Hausdorff distance between subsets of metric
spaces (see~\cite{Hau14}) emerges naturally in the context of $\V$-categories
from the construction discussed above.

\begin{proposition}
  \label{p:30}
  The lifting of the powerset functor $\ftP$ on $\SET$ to $\Cats{\V}$ with
  respect to $\bigwedge \colon \ftP\V \to \V$ sends a $\V$-category $(X,a)$ to
  $(\ftP X, \ftH a)$, where for all $A,B \subseteq X$,
  \[
    \ftH a(A,B) = \bigwedge_{y \in B}\bigvee_{x \in A} a(x,y).
  \]
\end{proposition}
\begin{proof}
  Let $(X,a)$ be a $\V$-category and $\ftP a$ the $\V$-category structure
  corresponding to the lifting  aforementioned.  That is, for every
  $A,B \in \ftP X$,
  \[
    \ftP a(A,B) = \bigwedge_{\psi \in \Cats{\V}(X,\V)} \hom(\bigwedge_{x \in A} \psi(x), \bigwedge_{y \in B} \psi(y)).
  \]

  First, observe that for every $u \in \V$ the function
  $\hom(u,-) \colon \V \to \V$ preserves infima and the map
  $\hom(-,u) \colon \V \to \V$ is antimonotone.

  Hence, for every $\V$-functor $\psi \colon (X,a) \to (\V, \hom)$,
  \[
   \ftH a(A,B) \leq \bigwedge_{y \in B} \bigvee_{x \in A} \hom(\psi(x), \psi(y))
               \leq \bigwedge_{y \in B} \hom(\bigwedge_{x \in A} \psi(x), \psi(y))
               = \hom(\bigwedge_{x \in A} \psi(x), \bigwedge_{y \in B} \psi(y)).
  \]
  Therefore, $\ftH a(A,B) \leq \ftP a(A,B)$.

  To see that the reverse inequality holds, consider the $\V$-functor
  $f \colon (X,a) \to (\V, \hom)$ below that is obtained by combining
  Propositions~\ref{p:9} and~\ref{p:10}.

  \begin{center}
    \begin{tikzcd}
      X\ar[r,"\mate{a}",near end]\ar[rr,"f",bend left=40] & {\V^A} \ar[r,"\bigvee", near start] & {\V}
    \end{tikzcd}
  \end{center}

  Therefore, as $\hom(-,u)$ is antimonotone,
  \[
    \ftP a(A,B) \leq \hom(\bigwedge_{y' \in A} f(y'), \bigwedge_{y \in B} f(y))
                \leq \hom(k, \bigwedge_{y \in B} \bigvee_{x \in A} a(x,y))
                = \bigwedge_{y \in B} \bigvee_{x \in A} a(x,y)
                = \ftH a(A,B).\qedhere
  \]
\end{proof}

\begin{remark}
  The notion of a (symmetric) distance between subsets of a metric space goes
  back to \cite{Pom05} and was made popular by its use in \cite{Hau14}. For more
  information on the history of this idea we refer to \cite{BT06}.
\end{remark}

\begin{corollary}
  \label{p:17}
  The lifting of the powerset functor to $\Cats{\V}$ of Proposition~\ref{p:30}
  preserves initial morphisms.
\end{corollary}

Another idea to tackle the problem of lifting an endofunctor $\ftF$ on $\SET$
to $\Cats{\V}$ is to consider first a lax extension $\widehat{\ftF}\colon
  \Rels{\V}\to\Rels{\V}$ of the functor $\ftF$ in the sense of \cite{Sea05}; that
is, to require
\begin{enumerate}
  \item $r\le r'\implies \widehat{\ftF} r\le\widehat{\ftF}r'$,
  \item $\widehat{\ftF} s\cdot\widehat{\ftF} r\le\widehat{\ftF} (s\cdot r)$,
  \item $\ftF f\le\widehat{\ftF}(f)$ and
    $(\ftF f)^\circ\le\widehat{\ftF}(f^\circ)$.
\end{enumerate}

It follows immediately (see \cite{Sea05}) that
\begin{equation*}
  \widehat{\ftF} (s\cdot f)=\widehat{\ftF} s\cdot \ftF f \quad\text{and}\quad
  \widehat{\ftF} (g^\circ\cdot r)=\ftF g^\circ\cdot\widehat{\ftF} r.
\end{equation*}

Then, based on this lax extension, the functor $\ftF \colon\SET\to\SET$ admits a
natural lifting to $\Cats{\V}$ (see~\cite{Tho09}): the functor
$\ftbF \colon \Cats{\V}\to\Cats{\V}$ sends a $\V$-category $(X,a)$ to $(\ftF
X,\widehat{\ftF}a)$. One advantage of this type of lifting is that allows us to
use the calculus of $\V$-relations. The following is a simple example.

\begin{proposition}
  \label{p:21}
  $\ftbF \colon\Cats{\V}\to\Cats{\V}$ preserves initial $\V$-functors.
\end{proposition}
\begin{proof}
  Let $f \colon(X,a)\to(Y,b)$ be a $\V$-functor with $a=f^\circ\cdot b\cdot
  f$. Then $\widehat{\ftF} a=\ftF f^\circ\cdot \widehat{\ftF} b\cdot \ftF f$.
\end{proof}

The result above generalises \cite[Theorem~5.8]{BBKK18}.

\begin{example}
  For a $\V$-relation $r \colon X \relto Y$, and subsets $A \subseteq X$, $B
    \subseteq Y$, the formula
  \[
    \bigwedge_{y \in B} \bigvee_{x \in A} r(x,y)
  \]
  defines a lax extension of the powerset functor on $\SET$ to $\Rels{\V}$
  (see~\cite{Sea05}). The corresponding lifting to $\Cats{\V}$ coincides with
  the one described in Proposition~\ref{p:30}. In particular, by
  Proposition~\ref{p:21}, we obtain another proof for the fact that this lifting
  preserves initial morphisms.
\end{example}

If we start with a monad $\mT=\monad$ on $\SET$, a \emph{lax extension} of
$\mT=\monad$ to $\Rels{\V}$ is a lax extension $\widehat{\ftT}$ of
the functor $\ftT$ to $\Rels{\V}$ such that
$m \colon \widehat{\ftT}\widehat{\ftT}\to\widehat{\ftT}$ and
$e \colon \ftId\to\widehat{\ftT}$ become op-lax:
\begin{align*}
  m_Y \cdot \widehat{\ftT}\widehat{\ftT} r & \le \widehat{\ftT} r \cdot m_X,
  & e_Y \cdot r \le \widehat{\ftT} r \cdot e_X
\end{align*}
for all $\V$-relations $r \colon X\relto Y$.

For a lax extension of a $\SET$-monad $\mT=\monad$ to $\Rels{\V}$, the
functions $e_X \colon X\to\ftT X$ and $m_X \colon \ftT\ftT X\to\ftT X$ become
$\V$-functors for each $\V$-category $X$, so that we obtain a monad on
$\Cats{\V}$. The Eilenberg--Moore algebras for this monad are triples
$(X,a,\alpha)$ where $(X,a)$ is a $\V$-category and $(X,\alpha)$ is an algebra
for the $\SET$-monad $\mT$ such that $\alpha \colon \ftT(X,a_0)\to(X,a_0)$ is a
$\V$-functor.  A map $f \colon X\to Y$ is a homomorphism
$f \colon (X,a,\alpha)\to (Y,b,\beta)$ of algebras precisely if $f$ preserves both
structures, that is, whenever $f \colon (X,a)\to(Y,b)$ is a $\V$-functor and
$f \colon (X,\alpha)\to(Y,\beta)$ is a $\mT$-homomorphism. For more information we
refer to \cite{Tho09,HST14}.

One possible way to construct lax extensions based on a (lax) $\mT$-algebra
structure $\xi \colon\ftT\V\to\V$ is devised in \cite{Hof07}: for every
$\V$-relation $r \colon X\times Y\to\V$ and for all $\fx\in \ftT X$ and $\fy\in
\ftT Y$,
\begin{equation*}
  \widehat{\ftT} r(\fx,\fy)=\bigvee\left\{\xi\cdot \ftT r(\fw)\;\Bigl\lvert\;\fw\in
    \ftT(X\times Y), \ftT\pi_1(\fw)=\fx,\ftT\pi_2(\fw)=\fy\right\}.
\end{equation*}
We note that $\widehat{\ftT}$ preserves the involution on $\Rels{\V}$, that is,
$\widehat{\ftT}(r^\circ)=(\widehat{\ftT} r)^\circ$ for all $\V$-relations
$r \colon X\relto Y$ (and we write simply $\widehat{\ftT} r^\circ$).

\begin{example}
  Consider the ultrafilter monad $\mU=\umonad$ on $\SET$, the quantale $\two$
  and the $\mU$-algebra
  \begin{equation*}
    \xi \colon \ftU 2\longrightarrow 2
  \end{equation*}
  sending every ultrafilter to its generating point. The category of algebras of
  the induced monad on $\Cats{\V}$ is the category $\ORDCH$ of (pre)ordered
  compact Hausdorff spaces introduced in \cite{Nac50} (see also \cite{Tho09}).
\end{example}

\section{Hausdorff polynomial functors on $\Cats{\V}$}
\label{sec:Hpf}

In this section we study a class of endofunctors on $\Cats{\V}$ that intuitively
is an analogue of the class of Kripke polynomial functors on $\SET$.  We begin by
describing a $\Cats{\V}$-counterpart of the powerset functor on $\SET$ that is
based on the upset functor on $\ORD$.

\subsection{The Hausdorff functor on $\Cats{\V}$}
\label{sec:hausdorff-functor}

We introduce now some $\V$-categorical versions of classical notions from
order theory. We start with the ``up-closure'' and ``down-closure'' of a subset.

\begin{definition}
  Let $(X,a)$ be a $\V$-category. For every $A\subseteq X$, put
  \[
    \upc^a A   = \{y \in X \mid k \le\bigvee_{x\in A}a(x,y) \}
    \qquad\text{and}\qquad
    \downc^a A = \{y \in X \mid k \le\bigvee_{x\in A}a(y,x) \}.
  \]
\end{definition}

As usual, we write $\upc^a x$ and $\downc^a x$ if $A=\{x\}$. We also observe
that $\upc^{a} A = \downc^{a^\circ} A$ which allows us to translate results
about $\upc^a$ to results about $\downc^a$, and \emph{vice versa}. Considering
the underlying ordered set $(X,\le)$ of $(X,a)$, we note that
\[
  \upc^\le A   \subseteq \upc^a A \qquad\text{and}\qquad
  \downc^\le A \subseteq \downc^a A
\]
for every $A\subseteq X$, with equality if $A$ is finite. To simplify
notation, we often write $\upc A$ and $\downc A$ whenever the corresponding
structure can be derived from the context.

\begin{remark}\label{d:rem:1}
  For an ordered set $X$, with $a$ denoting the $\V$-category structure induced
  by the order relation $\le$ of $X$, $\upc^\le A =\upc^a A$ and
  $\downc^\le A = \downc^a A$.
\end{remark}

\begin{lemma}
  For every $\V$-category $(X,a)$ and every $A\subseteq X$,
  \[
    A\subseteq \upc A,\quad \upc\upc A\subseteq \upc A, \quad
    A\subseteq \downc A, \quad \downc\downc A\subseteq\downc A.
  \]
\end{lemma}
\begin{proof}
  It follows immediately from the two defining properties of a $\V$-category.
\end{proof}

We call a subset $A\subseteq X$ of a $\V$-category $(X,a)$ \df{increasing}
whenever $A=\upc A$; likewise, $A$ is called \df{decreasing} whenever
$A=\downc A$. Clearly, $\upc A$ is the smallest increasing subset of $X$ which
includes $A$, and similarly for $\downc A$. For later use we record some simple
facts about increasing and decreasing subsets of a $\V$-category.

\begin{lemma}\label{d:lem:3}
  The intersection of increasing (decreasing) subsets of a $\V$-category is
  increasing (decreasing).
\end{lemma}

\begin{lemma}\label{d:lem:4}
  Let $f \colon X\to Y$ be a $\V$-functor. Then the following assertions hold.
  \begin{enumerate}
  \item For every increasing (decreasing) subset $B\subseteq Y$, $f^{-1}(B)$ is
    increasing (decreasing) in $X$.
  \item For every $A\subseteq X$, $f(\upc A)\subseteq \upc f(A)$ and
    $f(\downc A)\subseteq \downc f(A)$.
  \end{enumerate}
\end{lemma}

In contrast to the situation for ordered sets, the complement of an
increasing set is not necessarily decreasing. This motivates the following
notation.

\begin{definition}
  Let $(X,a)$ be a $\V$-category and $A\subseteq X$. Then $A$ is called
  \df{co-increasing} whenever $A^\complement$ is increasing, and $A$ is called
  \df{co-decreasing} whenever $A^\complement$ is decreasing.
\end{definition}

For a $\V$-category $(X,a)$, we consider the $\V$-category
\[
  \ftH X=\{A\subseteq X\mid A\text{ is increasing\}},
\]
equipped with
\[
  \ftH a(A,B)=\bigwedge_{y\in B}\bigvee_{x\in A}a(x,y),
\]
for all $A,B\in\ftH X$. It is well-known that the formula above defines indeed a
$\V$-category structure, not just on $\ftH X$ but even on the powerset $\ftP X$
of $X$ (for instance, see \cite{ACT10}).

Moreover, we have the following formulas.

\begin{lemma}
  \label{d:lem:6}
  Let $(X,a)$ be a $\V$-category. Then, for all $A,B\subseteq X$, the following
  assertions hold.
  \begin{enumerate}
  \item $k\le \ftH a(A,B)\iff B\subseteq \upc A$.
  \item $\ftH a(A,\upc B)=\ftH a(A,B)$ and $\ftH a(\upc A,B)= \ftH a(A,B)$.
  \end{enumerate}
\end{lemma}
\begin{proof}
  The first assertion is clear, and so are the inequalities
  $\ftH a(A,\upc B)\le \ftH a(A,B)$ and $\ftH a(\upc A,B)\le \ftH a(A,B)$. Furthermore,
  $\ftH a(A,B)\le \ftH a(A,B)\otimes \ftH a(B,\upc B)\le \ftH a(A,\upc B)$ and
  $\ftH a(\upc A,B)\le \ftH a(A,\upc A)\otimes \ftH a(\upc A,B)\le \ftH a(A,B)$.
\end{proof}

\begin{corollary}
  \label{d:cor:2}
  For every $\V$-category $(X,a)$, the $\V$-category $\ftH (X,a)$ is
  separated. Moreover, the underlying order is containment
  $\supseteq$.
\end{corollary}

For a $\V$-functor $f \colon(X,a)\to(X,a')$, the map
\[
  \ftH f \colon\ftH(X,a) \longrightarrow\ftH(Y,a')
\]
sends an increasing subset $A\subseteq X$ to $\upc f(A)$. Then, by
Lemma~\ref{d:lem:6},
\[
  \ftH a(A,B)\le \ftH a'(f(A),f(B))=\ftH a'(\upc f(A),\upc f(B))
\]
for all $A,B\in \ftH X$. Clearly, for the identity morphism $1_X \colon X\to X$
in $\Cats{\V}$, $\ftH(1_X)$ is the identity morphism on $\ftH X$. Moreover, for
all $f \colon X\to Y$ and $g \colon Y\to Z$ in $\Cats{\V}$ and $A\subseteq X$,
by Lemma~\ref{d:lem:4},
\[
  \upc g(f(A))\subseteq\upc g(\upc f(A))\subseteq \upc\upc g(f(A))\subseteq \upc
  g(f(A));
\]
which proves that the construction above defines a functor
$\ftH \colon\Cats{\V}\to\Cats{\V}$.

We note that this functor is naturally isomorphic to the ``Hausdorff functor''
$\mathcal{H}\colon\Cats{\V}\to\Cats{\V}$ of \cite{Stu10}, witnessed by the
family $(d_X \colon\ftH X \to \mathcal{H}X)_X$ where $A\in\ftH X$ is sent to the
presheaf $\ftH a(A,\{-\})$ on $X$. By \cite[Section~5.2]{Stu10}, each $d_X$ is
fully faithful and surjective; since $\ftH X$ is separated, $d_X$ is an
isomorphism in $\Cats{\V}$. For $f \colon(X,a)\to (Y,a')$ in $\Cats{\V}$,
$A\subseteq X$ increasing and $y\in Y$, we calculate
\begin{align*}
  \bigvee_{z\in A}a'(f(z),y)
  & \le \bigvee_{x\in X}\bigvee_{z\in A} (a(z,x)\otimes a'(f(x),y)) \\
  & \le \bigvee_{x\in X}\bigvee_{z\in A} (a'(f(z),f(x))\otimes a'(f(x),y)) \le
  \bigvee_{z\in A}a'(f(z),y)
\end{align*}
which proves that $(d_X)_X$ is indeed a natural transformation.  Consequently,
the functor $\ftH$ is part of a Kock--Z\"oberlein monad $\mH=\hmonad$ on
$\Cats{\V}$ where
\begin{align*}
  \coyoneda_X \colon X & \longrightarrow \ftH X, & \coyonmult_X \colon\ftH\ftH X & \longrightarrow\ftH X,\\
  x & \longmapsto \upc x & \mathcal{A} &\longmapsto \bigcup \mathcal{A}
\end{align*}
for all $\V$-categories $X$. Clearly, $\coyoneda$ corresponds to the unit of the
``Hausdorff monad'' of \cite{Stu10}; the following remark justifies the
corresponding claim regarding the multiplication.

\begin{remark}\label{d:rem:3}
  For all $\mathcal{A}\in\ftH\ftH X$,
  \[
    \bigcup \mathcal{A} =\{x\in X\mid \exists A\in\mathcal{A}\,.\,x\in A\}
                        =\{x\in X\mid \upc x\in\mathcal{A}\}=\coyoneda_X^{-1}(\mathcal{A}),
  \]
  therefore $\bigcup \mathcal{A}$ is indeed increasing.  Furthermore, we
  conclude that $\coyonmult_X\dashv\ftH \coyoneda_X$ in $\Cats{\V}$.
\end{remark}

\subsection{Coalgebras of Hausdorff polynomial functors on $\Cats{\V}$}
\label{sec:coalg-hpf-V-Cat}

The notion of Kripke polynomial functor is typically formulated in the context
of sets and functions. In this section we study an intuitive
$\Cats{\V}$-counterpart, where the Hausdorff functor on $\Cats{\V}$ takes the
role of the powerset functor on $\SET$. For previous studies of Kripke
polynomial functors see~\cite{Rut00a,BRS09,KKV04}.

\begin{definition}
  Let $\catX$ be a subcategory of $\Cats{\V}$ closed under finite limits and
  colimits such that the Hausdorff functor $\ftH \colon \Cats{\V} \to \Cats{\V}$
  restricts to $\catX$.  We call a functor \df{Hausdorff polynomial} on $\catX$
  if it belongs to the smallest class of endofunctors on $\catX$ that contains
  the identity functor, all constant functors and is closed under composition
  with $\ftH$, products and sums of functors.
\end{definition}

In the sequel, we will see that the category of coalgebras of a Hausdorff
polynomial functor on $\Cats{\V}$ is not necessarily complete. Nevertheless,
thanks to the next theorem, we are some small steps away from proving that
equalisers always exist.

\begin{theorem}[{\cite[Theorem~2.5.24]{Nor19}}]
  \label{p:19}
  Let $\ftF$ be an endofunctor over a cocomplete category $\catX$ that has an
  $(E,M)$-factorisation structure such that $E$ is contained in the class of
  $\catX$-epimorphisms and $\catX$ is $M$-wellpowered.  If $\ftF$ sends
  morphisms in $M$ to morphisms in $M$, then $\CoAlg(\ftF)$ has equalisers.
\end{theorem}

\begin{corollary}
  \label{p:18}
  The Hausdorff functor $\ftH \colon \Cats{\V} \to \Cats{\V}$ preserves initial
  morphisms.
\end{corollary}
\begin{proof}
  Let $f \colon (X,a) \to (Y,b)$ be an initial morphism in $\Cats{\V}$.
  Consider the map $\upc (-) \colon \ftP (Y,b) \to \ftH (Y,b)$ defined by
  $A \mapsto \upc fA$.
  By Lemma~\ref{d:lem:6}, $\upc (-)$ is an initial morphism in $\Cats{\V}$.
  Therefore, by Corollary~\ref{p:17}, we can express $\ftH f$ as the following
  composition of initial morphisms
  \begin{equation*}
    \begin{tikzcd}[row sep=large, column sep=large,ampersand replacement=\&]
      \ftH (X,a) \ar[r,"\ftH f"] \ar[d]  \& \ftH (Y,b) \\
      \ftP (X,a) \ar[r, "\ftP f"']       \& \ftP (Y,b) \ar[u, "\upc (-)"']
   \end{tikzcd}.
 \end{equation*}
\end{proof}

\begin{proposition}
  \label{p:20}
  The Hausdorff functor $\ftH \colon \Cats{\V} \to \Cats{\V}$ preserves initial
  monomorphisms.
\end{proposition}
\begin{proof}
  We already know from Corollary~\ref{p:18} that $\ftH$ preserves initial
  morphisms, and from Corollary~\ref{d:cor:2} that the image by $\ftH$ of every
  $\V$-category is separated.  Therefore, $\ftH$ preserves initial
  monomorphisms.
\end{proof}

\begin{proposition}
  The category of coalgebras of a Hausdorff polynomial functor
  $\ftH \colon \Cats{\V} \to \Cats{\V}$ has equalisers.
\end{proposition}
\begin{proof}
  Being a topological category over $\SET$, the category $\Cats{\V}$ is
  (surjective, initial mono)-structured and satisfies all conditions necessary
  to apply Theorem~\ref{p:19}. By Proposition~\ref{p:20}, the Hausdorff functor
  preserves initial monomorphisms and the remaining cases follow from standard
  arguments.  Therefore, we can apply Theorem~\ref{p:19}.
\end{proof}

In the remainder of the section, we show that the Hausdorff functor does not
admit a terminal coalgebra. This part is inspired by \cite{DG62}.

Given elements $x,y$ of a $\V$-category $(X,a)$, we write $x \prec y$ if
$k \leq a(x,y)$ and $a(y,x) = \bot$, and we denote by $\eupc x$ the set
$\{y \in X \mid x \prec y \}$.

\begin{proposition}
  \label{p:8}
  Let $(X,a)$ be a $\V$-category.  Then, for every $x,y \in X$, the following
  assertions hold.
  \begin{enumerate}
    \item The set $\eupc x$ is increasing.
    \item $\upc x \prec \eupc x$ in $\ftH(X,a)$.
    \item For every initial $\V$-functor $(X,a) \to (Y,b)$, if $x \prec y$ then
      $fx \prec fy$.
  \end{enumerate}
\end{proposition}
\begin{proof}
  The set $\eupc x$ is the intersection of the increasing sets $\upc x$ and
  $a(-,x)^{-1}\{\bot\}$.  Regarding the second affirmation, observe that
  $\ftH a(\eupc x, \upc x) \leq \bigvee_{y \in \eupc x} a(y,x) = \bot$. The
  third affirmation is trivial.
\end{proof}

\begin{theorem}
  \label{d:thm:1}
  Let $\V$ be a non-trivial quantale, and $(X,a)$ a $\V$-category.
  A morphism of type $\ftH(X,a) \to (X,a)$ cannot be an embedding.
\end{theorem}
\begin{proof}
  Suppose that there exists an embedding $\phi \colon \ftH(X,a) \to (X,a)$.
  We will see that this implies that there exists $x \in X$ such that
  $\upc x = \eupc x$, which is a contradiction as $\V$ is non-trivial.

  Since $\ftH X$ is a complete lattice the map
  $\coyoneda_X \cdot \phi \colon \ftH X \to \ftH X$
  has a greatest fixed point $A$ that is given by
  \[
    \bigvee \{ I \in \ftH X \mid I \leq \upc\phi(I) \}.
  \]

  We claim that $x = \phi(A)$ has the desired property. The morphism $\phi$ is
  initial and $\upc x \prec \eupc x$, hence, by Proposition~\ref{p:8},
  $x = \phi(\upc x) \prec \phi(\eupc x)$ and, consequently,
  $\eupc x \leq \upc\phi(\eupc x)$.  Therefore, $\eupc x \leq \upc x$ because
  $\upc x$ is the greatest fixed point.
\end{proof}
\begin{corollary}
  Let $\V$ be a non-trivial quantale.
  The Hausdorff functor $\ftH \colon \Cats{\V} \to \Cats{\V}$ does not admit fixed
  points.
\end{corollary}

\begin{remark}
 If $\V$ is trivial, that is $\V = \{k\}$, then $\ftH \colon \SET \to \SET$ is the
 functor that sends every set $X$ to the set $\{X\}$.
 Therefore, the fixed points of $\ftH \colon \SET \to \SET$ are the terminal objects.
\end{remark}

\begin{corollary}
  Let $\V$ be a non-trivial quantale. The Hausdorff functor $\ftH
    \colon\Cats{\V} \to \Cats{\V}$ does not admit a terminal coalgebra, neither
  does any possible restriction to a full subcategory of $\Cats{\V}$.
\end{corollary}

\begin{example}
  In particular, the (non-symmetric) Hausdorff functor on $\MET$ does not admit
  a terminal coalgebra, and the same applies to its restriction to the full
  subcategory of compact metric spaces.  Passing to the symmetric version does
  not remedy the situation. Here, for a symmetric compact metric space $(X,a)$,
  we consider now the metric $\ftH a$ defined by
  \begin{equation*}
    \ftH a(A,A')
    =\max\left\{\sup_{x\in A}\inf_{x'\in A'}a(x,x'), \sup_{x\in A'}\inf_{x'\in A}a(x',x)\right\}
  \end{equation*}
  on the set $\ftH X$ of all closed subsets.  Note that
  $\ftH a(\varnothing,A)=\infty$, for every non-empty subset $A\subseteq
  X$. Then, if $s \colon (\ftH X,\ftH a)\to (X,a)$ is an isomorphism, we
  construct recursively a sequence $(x_n)_{n\in\N}$ in $X$ as follows:
  \begin{equation*}
    x_0=s(\varnothing)\quad\text{and}\quad x_{n+1}=s(\{x_n\}).
  \end{equation*}
  Then $a(x_m,x_l)=\infty$, for all $m,k\in\N$ with $m \neq k$; which
  contradicts compactness of $(X,a)$.
  % Note: $X$ is separated and every singleton subset of $X$ is closed.
\end{example}

\section{Hausdorff polynomial functors on $\CatCHs{\V}$}
\label{sec:hausd-polyn-funct}

In Section~\ref{sec:coalg-hpf-V-Cat} we saw that the image of a $\V$-category
under the Hausdorff functor $\ftH$ on $\Cats{\V}$ has ``too many'' elements for
$\ftH$ to admit a terminal coalgebra. To filter them, in this section we add a
topological component to our study of $\Cats{\V}$.

\subsection{Adding topology}
\label{sec:adding-topology}

To ``add topology'', we use the ultrafilter monad $\mU=\umonad$ on $\SET$.
Furthermore:

\begin{assumption}\label{d:ass:4}
  Throughout this section we assume that $\V$ is completely distributive
  quantale (see \cite{Ran52,FW90}).
\end{assumption}

Then
\begin{equation*}
  \xi \colon \ftU \V \longrightarrow\V, \quad
  \fv \longmapsto \bigwedge_{A\in\fv}\bigvee A
\end{equation*}
is the structure of an $\mU$-algebra on $\V$, and represents the convergence of
a compact Hausdorff topology. Therefore, as discussed at the end of
Section~\ref{sec:gener-cons}, we obtain a lax extension of the ultrafilter monad
to $\Rels{\V}$ that induces a monad on $\Cats{\V}$.  Its algebras are
$\V$-categories equipped with a \emph{compatible} compact Hausdorff topology
(see~\cite{Tho09,HR18});  we call them \df{$\V$-categorical compact Hausdorff
spaces}, and denote the corresponding Eilenberg--Moore category by

\begin{equation*}
  \CatCHs{\V}.
\end{equation*}

Then we have a natural forgetful functor
\begin{equation*}
  \CatCHs{\V}\longrightarrow \ORDCH
\end{equation*}
sending $(X,a,\alpha)$ to $(X,\le,\alpha)$ where
\begin{equation*}
  x\le y\quad\text{whenever}\quad k\le a(x,y).
\end{equation*}
Moreover, $(\V,\hom,\xi)$ is a $\V$-categorical compact Hausdorff space with
underlying ordered compact Hausdorff space $(\V,\le,\xi)$, where $\leq$ is the
order of $\V$. We denote by $\xi_\le$ the induced stably compact topology. We
provide now some information on the topologies of $\V$.

\begin{remark}
  \label{d:rem:6}
  Since $\V$ is in particular a continuous lattice, the convergence $\xi$ is the
  convergence of the Lawson topology of $\V$ (see
  \cite[Proposition~VII-3.10]{GHK+03}). A subbasis for this topology is given by
  the sets
  \begin{equation*}
    \{u\in\V\mid v\ll u\}\quad\text{and}\quad \{u\in\V\mid v\nleq u\}\qquad (v\in\V),
  \end{equation*}
  where $\ll$ denotes the way-below relation of $\V$. Furthermore, by
  \cite[Proposition~2.3.6]{AJ95}, the sets
  \begin{equation*}
    \{u\in\V\mid v\ll u\}\qquad (v\in\V)
  \end{equation*}
  form a basis for the Scott topology of $\V$. By the proof of
  \cite[Lemma~V-5.15]{GHK+03}, the sets
  \begin{equation*}
    \{u\in\V\mid v\nleq u\} = (\upc v)^\complement\qquad (v\in\V)
  \end{equation*}
  form a subbasis of the dual of Scott topology of $\V$, which is precisely
  $\xi_\le$.

  Since, moreover, $\V$ is (ccd), we have the following.
  \begin{itemize}
  \item By \cite[Lemma~VII-2.7]{GHK+03} and \cite[Proposition~VII-2.10]{GHK+03},
    the Lawson topology of $\V$ coincides with the Lawson topology of $\V^\op$,
    and the set
    \begin{equation*}
      \{ \upc u \mid u \in \V \} \cup \{ \downc u \mid u \in \V \}
    \end{equation*}
    is a subbasis for the closed sets of this topology which is known as the
    interval topology.
  \item Therefore the Scott topology of $\V$ coincides with the dual of the
    Scott topology of $\V^\op$; in particular, the sets $\downc v$ ($v\in\V$)
    form a subbasis for the closed sets of the Scott topology of $\V$.
  \item Finally, also the sets
    \begin{equation*}
      \{u\in\V\mid v\lll u\}\qquad (v\in\V)
    \end{equation*}
    form a subbasis of the Scott topology of $\V$.
  \end{itemize}
\end{remark}

We aim now at $\V$-categorical generalisations of some results of \cite{Nac50}
regarding ordered compact Hausdorff spaces. Firstly, we recall
\cite[Proposition~3.22]{HR18}:

\begin{proposition}\label{d:prop:4}
  For a $\V$-category $(X,a)$ and a $\mU$-algebra $(X,\alpha)$ with the same
  underlying set $X$, the following assertions are equivalent.
  \begin{tfae}
  \item $\alpha \colon \ftU(X,a)\to(X,a)$ is a $\V$-functor.
  \item $a \colon (X,\alpha)\times(X,\alpha)\to(\V,\xi_\le)$ is continuous.
  \end{tfae}
\end{proposition}

\begin{corollary}\label{d:cor:1}
  For a $\V$-category $(X,a)$ and a $\mU$-algebra $(X,\alpha)$ with the same
  underlying set $X$, $(X,a,\alpha)$ is a $\V$-categorical compact Hausdorff
  space if and only if, for all $x,y\in X$ and $u\in\V$ with $u\not\leq a(x,y)$,
  there exist neighbourhoods $V$ of $x$ and $W$ of $y$ so that, for all
  $x'\in V$ and $y'\in W$, $u\not\le a(x',y')$.
\end{corollary}
\begin{proof}
  It follows from the fact that the sets $(\upc u)^\complement$ ($u\in\V$) form
  a subbasis for the topology $\xi_\le$ on $\V$ (see Remark~\ref{d:rem:6}).
\end{proof}

We consider now the full subcategory $\CatCHs{\V}_\sep$ of $\CatCHs{\V}$ defined
by the separated $\V$-categorical compact Hausdorff spaces; i.e.\ those spaces
where the underlying $\V$-category is separated. The results above imply that
the separated reflector $\ftR \colon\Cats{\V}\to\Cats{\V}_\sep$ lifts to a
functor $\ftS \colon \CatCHs{\V}\to\CatCHs{\V}_\sep$ which is left adjoint to
the inclusion functor $\CatCHs{\V}_\sep\to\CatCHs{\V}$. In fact, for a
$\V$-categorical compact Hausdorff space $(X,a,\alpha)$, the equivalence
relation $\sim$ on $X$ is closed in $X\times X$ with respect to the product
topology, therefore the quotient topology on $X/{\sim}$ is compact Hausdorff
and, with $p \colon X\to X/{\sim}$ denoting the projection map, the diagram
\begin{equation*}
  \begin{tikzcd}[row sep=large, column sep=large,ampersand replacement=\&]
    X\times X
    \ar[r,"p\times p"]\ar[dr,"a"']
    \& X/{\sim}\times X/{\sim}
    \ar[d,"\widetilde{a}"] \\
    \& \V
  \end{tikzcd}
\end{equation*}
commutes. Consequently, the $\V$-category $(X/{\sim},\widetilde{a})$ together
with the quotient topology on $X/{\sim}$ is a $\V$-categorical compact Hausdorff
space. In contrast to Remark~\ref{d:rem:5}, now we have the following result.

\begin{proposition}
  \label{p:24}
  The functor $\ftS \colon \CatCHs{\V}\to\CatCHs{\V}_\sep$ preserves codirected
  limits.
\end{proposition}
\begin{proof}
  Let $D \colon I \to\CatCHs{\V}$ be a codirected diagram with limit cone
  $(\pi_i \colon X\to X_i)_{i\in I}$. Let $(\rho_i \colon L\to \ftS X_i)_{i\in I}$
  be a limit cone of $\ftS D$ in $\CatCHs{\V}_\sep$ and $q \colon \ftS X\to L$ be the
  canonical comparison map. By Corollary~\ref{d:cor:4}, $q$ is initial with
  respect to the forgetful functor $\CatCHs{\V}_\sep\to\COMPHAUS$. Since the
  diagram
  \begin{equation*}
    \begin{tikzcd}[row sep=large, column sep=large,ampersand replacement=\&]
      X
      \ar[r,"p"]\ar[d,"\pi_i"']
      \& \ftS X \ar[r,"q"] \&  L
      \ar[d,"\rho_i"] \\
      X_i
      \ar[rr,"p_i"']
      \&\& \ftS X_i
    \end{tikzcd}
  \end{equation*}
  commutes, $q\cdot p$ is surjective by \cite[I.9.6, Corollary 2]{Bou66} hence
  $q$ is surjective and therefore an isomorphism in $\CatCHs{\V}_\sep$.
\end{proof}

Besides the compact Hausdorff space $(X,\alpha)$, we also consider the stably
compact topology $a_\le$ induced by $\alpha$ and the underlying order of $a$, as
well as the dual space $(X,\alpha_\le)^\op$ of $(X,\alpha_\le)$. We remark that
the identity map $1_X \colon X\to X$ is continuous of types
\begin{equation*}
  (X,\alpha) \longrightarrow (X,\alpha_\le)
  \qquad\text{and}\qquad
  (X,\alpha) \longrightarrow (X,\alpha_\le)^\op.
\end{equation*}
Therefore a subset $A\subseteq X$ of $X$ is open (closed) in $(X,\alpha)$ if it
is open (closed) in $(X,\alpha_\le)$ or in $(X,\alpha_\le)^\op$. Moreover, every
closed subset of $(X,\alpha)$ is compact in $(X,\alpha_\le)$ and in
$(X,\alpha_\le)^\op$.

\begin{corollary}\label{d:cor:3}
  Let $(X,a,\alpha)$ be a $\V$-categorical compact Hausdorff space. Then also
  \begin{equation*}
    a \colon (X,\alpha_\le)^\op\times (X,\alpha_\le) \longrightarrow (\V,\xi_\le)
  \end{equation*}
  is continuous. Hence, for all $x,y\in X$ and $u\in\V$ with $u\not\leq a(x,y)$,
  there exist a neighbourhood $V$ of $x$ in $(X,\alpha_\le)^\op$ and a
  neighbourhood $W$ of $y$ in $(X,\alpha_\le)$ so that, for all $x'\in V$ and
  $y'\in W$, $u\not\le a(x',y')$.
\end{corollary}
\begin{proof}
  Follows from the facts that $a \colon X\times X\to\V$ is continuous of type
  $(X,\alpha)\times(X,\alpha)\to(\V,\xi_\le)$ and monotone of type
  $(X,\le)^\op\times(X,\le)\to(\V,\le)$.
\end{proof}

\begin{remark}\label{d:rem:4}
  The result above allows us to construct some useful continuous maps. For
  instance, for $A\subseteq X$ compact in $(X,\alpha_\le)^\op$, the map
  $a \colon A\times X\to \V$ is continuous where we consider on $A$ the subspace
  topology. Therefore the composite arrow
  \begin{center}
    \begin{tikzcd}
      X\ar[r,"\mate{a}",near end]\ar[rr,"\upc^a_A",bend left=40] & {\V^A}
      \ar[r,"\bigvee", near start] & {\V}
    \end{tikzcd}
  \end{center}
  is continuous of type $(X,\alpha_\le)\to(\V,\xi_\le)$. Note that
  \begin{equation*}
    \upc^a_A(x)=\bigvee\{a(z,x)\mid z\in A\},
  \end{equation*}
  for every $x\in X$. Similarly, for $A\subseteq X$ compact in $(X,\alpha_\le)$,
  we obtain a continuous map
  $\downc_A^a \colon (X,\alpha_\le)^\op\to(\V,\xi_\le)$ sending $x\in X$ to
   \begin{equation*}
    \downc^a_A(x)=\bigvee\{a(x,z)\mid z\in A\}.
  \end{equation*}
\end{remark}

\begin{lemma}\label{d:lem:2}
  Let $(X,a,\alpha)$ be a $\V$-categorical compact Hausdorff space and
  $A\subseteq X$. Then the following assertions hold.
  \begin{enumerate}
  \item\label{d:item:6} If $A$ is compact in $(X,\alpha_\le)^\op$, then
    $\upc^a A$ is closed in $(X,\alpha_\le)$ and therefore also in $(X,\alpha)$.
  \item\label{d:item:7} If $A$ is compact in $(X,\alpha_\le)$, then $\downc^a A$
    is closed in $(X,\alpha_\le)^\op$ and therefore also in $(X,\alpha)$.
  \end{enumerate}
  In particular, if $A$ is closed in $(X,\alpha)$, then $\upc^a A$ and
  $\downc^a A$ are closed in $(X,\alpha_\le)$ and hence also in $(X,\alpha)$.
\end{lemma}
\begin{proof}
  Use the maps to Remark~\ref{d:rem:4} and observe that
  \begin{equation*}
    \upc^a A=(\upc^a_A)^{-1}(\upc k)
    \qquad\text{and}\qquad
    \downc^a A=(\downc^a_A)^{-1}(\upc k).\qedhere
  \end{equation*}
\end{proof}

From now on we assume the following condition.
\begin{assumption}\label{d:ass:1}
  The subset
  \begin{equation*}
    \tbdown k=\{u\in\V\mid u\lll k\}
  \end{equation*}
  of $\V$ is directed; which implies in particular that $k\neq\bot$. A quantale
  satisfying this condition is called \emph{value quantale} in \cite{Fla97},
  whereby in \cite{HR18} the designation \emph{$k$ is approximated} is used.
\end{assumption}

The assumption above implies some further pleasant properties of $\V$, as we
recall next.
\begin{lemma}
  The $\otimes$-neutral element $k$ satisfies the conditions
  \begin{equation*}
    (k\le u\vee v)\implies ((k\le u) \text{ or }(k\le v)),
  \end{equation*}
  for all $u,v\in\V$, and
  \begin{equation*}
    k\le\bigvee_{u\lll k}u\otimes u.
  \end{equation*}
\end{lemma}
\begin{proof}
  See \cite[Theorem~1.12]{Fla92} and \cite[Remark~4.21]{HR13}.
\end{proof}

\begin{lemma}\label{d:lem:5}
  Let $(X,a,\alpha)$ be a $\V$-categorical compact Hausdorff space and
  $A,B\subseteq X$ so that $A\cap B=\varnothing$, $A$ is increasing and compact
  in $(X,\alpha_\le)^\op$ and $B$ is compact in $(X,\alpha_\le)$. Then there
  exists some $u\lll k$ so that, for all $x\in A$ and $y\in B$,
  $u\not\le a(x,y)$.
\end{lemma}
\begin{proof}
  Let $y\in B$. Since $A$ is increasing and $y\notin A$, there is some
  $v_y\lll k$ so that
  \[
    v_y\not\le\bigvee_{x\in A}a(x,y).
  \]
  Hence, by Corollary~\ref{d:cor:3}, for every $x \in A$ there exists $U_{xy}$
  open in $(X, \alpha_\leq)^\op$ and $W_{xy}$ open in $(X, \alpha_\leq)$ such
  that $y \in W_{xy}$ and
  \[
    \forall x'\in U_{xy},y'\in W_{xy}\,.\,v_y\not\le a(x',y').
  \]
  Therefore, by compactness of $A$, there exists an open subset $U_y$ in
  $(X,\alpha_\le)^\op$ and an open subset $W_y$ in $(X,\alpha_\le)$ such that
  $A\subseteq U_y$, $y\in W_y$ and
  \[
    \forall x'\in U_y,y'\in W_y\,.\,v_y\not\le a(x',y').
  \]
  Then $B\subseteq \bigcup\{W_y\mid y\in X,y\notin V\}$ and, since $B$ is
  compact, there are finitely many elements $y_1,\dots,y_n\in B$ with
  \[
    B\subseteq W_{y_1}\cup\dots\cup W_{y_n}.
  \]
  Put $u=v_{y_1}\vee\dots\vee v_{y_n}$. Then $u\lll k$ since $\tbdown k$ is
  directed; moreover, $u\not\le a(x,y)$, for all $x\in A$ and $y\in B$.
\end{proof}

\begin{lemma}\label{d:lem:8}
  Let $A\subseteq\V$ be compact subset in $(\V,\xi_\le)$. If $k\le \bigvee A$,
  then there is some $u\in A$ with $k\le u$.
\end{lemma}
\begin{proof}
  Assume that $\upc k\cap \downc A=\varnothing$. Since $\upc k$ is increasing
  and compact in $(\V,\xi_\le)^\op$ and $\downc A$ is compact in $(\V,\xi_\le)$,
  by Lemma~\ref{d:lem:5}, there is some $u\lll k$ so that, for all $v\in A$,
  $u\not\le\hom(k,v)=v$. Therefore $k\not\le\bigvee A$.
\end{proof}

Combining Remark~\ref{d:rem:4} with Lemma~\ref{d:lem:8}, we obtain:

\begin{lemma}\label{d:lem:9}
  Let $(X,a,\alpha)$ be a $\V$-categorical compact Hausdorff spaces with
  underlying order $\le$. Then, for every compact subset $A\subseteq X$ of
  $(X,\alpha_\le)^\op$, $\upc^a A=\upc^\le A$; likewise, for every compact
  subset $A\subseteq X$ of $(X,\alpha_\le)$, $\downc^a A=\downc^\le A$. In
  particular, for every closed subset $A\subseteq X$ of $(X,\alpha)$,
  $\downc^a A=\downc^\le A$ and $\upc^a A=\upc^\le A$.
\end{lemma}

Thanks to Lemma~\ref{d:lem:9} we can transport several well-known result for
ordered compact Hausdorff spaces to metric compact Hausdorff spaces.

\begin{lemma}
  \label{d:lem:14}
  Let $(X,a,\alpha)$ be a $\V$-categorical compact Hausdorff space,
  $A\subseteq X$ closed and $W\subseteq X$ open and co-increasing with
  $A\subseteq W$. Then $\downc A\subseteq W$.
\end{lemma}
\begin{proof}
  Apply Lemma~\ref{d:lem:5} to
  $W^\complement\subseteq A^\complement$.
\end{proof}

The following result is \cite[Proposition~5]{Nac65}.

\begin{proposition}\label{d:prop:1}
  Let $(X,a,\alpha)$ be a $\V$-categorical compact Hausdorff space,
  $A\subseteq X$ closed and increasing and $V\subseteq X$ open with
  $A\subseteq V$. Then there exists $W\subseteq X$ open and co-decreasing with
  $A\subseteq W\subseteq V$.
\end{proposition}

\begin{theorem}\label{d:thm:3}
  Let $(X,a,\alpha)$ be a $\V$-categorical compact Hausdorff space,
  $A\subseteq X$ closed and decreasing and $B\subseteq X$ closed and increasing
  with $A\cap B=\varnothing$. Then there exist $V\subseteq X$ open and
  co-increasing and $W\subseteq X$ open and co-decreasing with
  \[
    A\subseteq V,\quad B\subseteq W,\quad V\cap W=\varnothing.
  \]
\end{theorem}
\begin{proof}
  See \cite[Theorem~4]{Nac65}.
\end{proof}

For a $\V$-categorical compact Hausdorff space $X=(X,a,\alpha)$, we put
\[
  \ftH X=\{A\subseteq X\mid A\text{ is closed and increasing\}}
\]
and consider on $\ftH X$ the restriction of the Hausdorff structure
$\ftH a$ to $\ftH X$ and the \df{hit-and-miss topology}, that is, the topology
generated by the sets
\[
  V^\Diamond=\{A\in\ftH X \mid A\cap V \neq \varnothing\}\qquad (\text{$V$ open,
    co-increasing})
\]
and
\[
  W^\Box=\{A\in\ftH X \mid A\subseteq W\}\qquad (\text{$W$ open,
    co-decreasing}).
\]

Note that, by Lemma~\ref{d:lem:9}, the topological part of $\ftH X$ coincides
with the Vietoris topology for the underlying ordered compact Hausdorff
space. In particular:

\begin{proposition}\label{d:prop:5}
  For every $\V$-categorical compact Hausdorff space $X$, the hit-and-miss
  topology on $\ftH X$ is compact and Hausdorff.
\end{proposition}

\begin{proposition}\label{d:prop:6}
  For every $\V$-categorical compact Hausdorff space $X$, $\ftH X$ equipped with
  the hit-and-miss topology and the Hausdorff structure is a $\V$-categorical
  compact Hausdorff space.
\end{proposition}
\begin{proof}
  Consider a $\V$-categorical compact Hausdorff space $(X,a,\alpha)$. To
  establish the compatibility between the topology and the Hausdorff structure, we use
  Corollary~\ref{d:cor:1}. Let $A,B\in\ftH X$ and $u\in\V$. Assume
  $u\not\le \ftH a(A,B)$. Since $\V$ is (ccd), there is some $v\lll u$ with
  $v\not\le \ftH a(A,B)$. Hence, there is some $y\in B$ with
  $v\not\le \bigvee_{x\in A}a(x,y)$. Therefore $v\not\le a(x,y)$ for all
  $x\in A$. By Corollary~\ref{d:cor:1} and compactness of $A$, there exist open
  subsets $U,V\subseteq X$ with $A\subseteq U$, $y\in V$, and
  $v\not\le a(x',y')$ for all $x'\in U$ and $y'\in V$; by
  Proposition~\ref{d:prop:1}, we may assume that $U$ is co-decreasing and $V$ is
  co-increasing. We conclude that
  \[
    A\in U^\Box,\quad B\in V^\Diamond,\quad u\not\le \ftH a(A',B')
  \]
  for all $A'\in U^\Box$ and $B'\in V^\Diamond$.
\end{proof}

\begin{lemma}\label{d:lem:7}
  Let $f \colon X\to Y$ be in $\CatCHs{\V}$. Then the map
  \[
    \ftH f \colon \ftH X \longrightarrow\ftH Y,\,A\longmapsto\upc f(A)
  \]
  is continuous and a $\V$-functor.
\end{lemma}

Clearly, the construction of Lemma~\ref{d:lem:7} defines a functor
\[
  \ftH \colon\CatCHs{\V} \longrightarrow\CatCHs{\V}.
\]
Moreover:

\begin{proposition}
  \label{p:6}
  The diagrams
  \begin{center}
    \begin{tikzcd}[row sep=large, column sep=large]
      \ORDCH \ar[r,"\ftH"]\ar[d,""'] & \ORDCH \ar[d,""] \\
      \CatCHs{\V} \ar[r,"\ftH"']     & \CatCHs{\V}
    \end{tikzcd}
    \qquad
    \begin{tikzcd}[row sep=large, column sep=large,ampersand replacement=\&]
      \CatCHs{\V} \ar[r,"\ftH"]\ar[d,""'] \& \CatCHs{\V} \ar[d,""] \\
      \ORDCH \ar[r,"\ftH"']               \& \ORDCH
    \end{tikzcd}
  \end{center}
  of functors commutes.
\end{proposition}

\begin{remark}
  \label{p:7}
  Despite the commutative diagrams of Proposition~\ref{p:6} above, we cannot
  apply Theorem~\ref{p:3} because, in general, the functors are not topological.
  In fact, even the functor $\MET \to \ORD$ fails to be fibre-complete since a
  metric $d$ in the fibre of $\{0 \leq 1 \}$ is completely determined by the
  value $d(1,0) \in (0, \infty]$. On the other hand, the functor
  $\CatCHs{\V} \to \COMPHAUS$ is topological and it is easy to see that the
  Hausdorff $\V$-category structure is compatible with the classical Vietoris
  topology on $\COMPHAUS$. Therefore, Theorem~\ref{p:3} tell us that equipping
  the Vietoris space on $\COMPHAUS$ with the Hausdorff structure yields a
  ``powerset kind of'' functor on $\CatCHs{\V}$ that, in some sense, disregards
  the $\V$-category structure of the objects, but whose category of coalgebras
  is (co)complete.
\end{remark}

\begin{theorem}
  \label{d:thm:4}
  The functor $\ftH$ is part of a Kock--Z\"oberlein monad $\mH=\hmonad$ on
  $\CatCHs{\V}$; for every $X$ in $\CatCHs{\V}$, the components $\coyoneda_X$ and $\coyonmult_X$ are given by
  \begin{align*}
    \coyoneda_X \colon X & \longrightarrow \ftH X, & \coyonmult_X \colon\ftH\ftH X & \longrightarrow\ftH X.\\
    x & \longmapsto \upc x & \mathcal{A} &\longmapsto \bigcup \mathcal{A}
  \end{align*}
\end{theorem}

We recall from \cite{HT10} that to every $\V$-category one can associate a
canonical closure operator which generalises the classic topology associated to
a metric space.

\begin{proposition}
  For every $\V$-category $(X,a)$, $A\subseteq X$ and $x\in X$,
  \begin{equation*}
    x\in\overline{A}\iff k\le \bigvee_{z\in A}a(x,z)\otimes a(z,x).
  \end{equation*}
  Moreover, the closure operator $\overline{(-)}$ is topological for every
  $\V$-category and defines a functor
  \begin{equation*}
    \ftL_\V \colon \Cats{\V} \longrightarrow\TOP
  \end{equation*}
  which commutes with the forgetful functors to $\SET$. Moreover,
  $L_\V(X)=L_\V(X^\op)$ for every $\V$-category $X$.
\end{proposition}
\begin{proof}
  See \cite{HT10}.
\end{proof}

Recall that we assume $\tbdown k$ to be directed.

\begin{proposition}
  \begin{enumerate}
  \item For every $\V$-category $(X,a)$, the topology of $\ftL_\V(X,a)$ is
    generated by the \emph{left centered balls}
    \begin{equation*}
      \OLB(x,u)=\{y\in X\mid u\lll a(x,y)\}\qquad (x\in X,\,u\lll k)
    \end{equation*}
    end the \emph{right centered balls}
    \begin{equation*}
      \ORB(x,u)=\{y\in X\mid u\lll a(y,x)\}\qquad (x\in X,\,u\lll k).
    \end{equation*}
  \item For every separated $\V$-category $(X,a)$, the space $\ftL(X,a)$ is
    Hausdorff.
  \end{enumerate}
\end{proposition}
\begin{proof}
  Regarding first statement, see \cite[Remark~4.21]{HR13} and \cite{Fla92}. The
  proof of the second statement is analogous to the one for classic metric
  spaces. In fact, assume that $(X,a)$ is separated and let $x,y\in X$ with
  $x\neq y$. Without loss of generality, we may assume that $k\nleq
  a(x,y)$. Hence, there is some $u\lll k$ with $u\nleq a(x,y)$. Take $v,w\lll k$
  with $u\le v\otimes w$. Then
  \begin{equation*}
    \OLB(x,v)\cap \ORB(y,w)=\varnothing
  \end{equation*}
  since, if $z\in\OLB(x,v)\cap \ORB(y,w)$, then
  \begin{equation*}
    u\le v\otimes w\le a(x,z)\otimes a(z,y)\le a(x,y),
  \end{equation*}
  a contradiction.
\end{proof}

Until the end of this section we require also the following condition.
\begin{assumption}\label{d:ass:2}
  For all $u,v\in\V$,
  \begin{equation*}
    (k \leq u \otimes v)
    \implies (k \leq u\quad\text{and}\quad k \leq v).
  \end{equation*}
\end{assumption}

\begin{remark}\label{d:rem:2}
  For every subset $A\subseteq X$ of a $\V$-category $(X,a)$,
  \begin{equation*}
    \overline{A}\subseteq \upc A\cap\downc A.
  \end{equation*}
  In fact, if $x\in\overline{A}$, then
  \begin{equation*}
    k\le \bigvee_{z\in A}(a(x,z)\otimes a(z,x))
    \le \left(\bigvee_{z\in A}a(x,z)\right)\otimes
    \left(\bigvee_{z\in A}a(z,x)\right)
  \end{equation*}
  and therefore $k\le \bigvee_{z\in A}a(x,z)$ and $k\le\bigvee_{z\in
    A}a(z,x))$. In particular, every increasing and every decreasing subset of
  $X$ are closed with respect to the closure operator of $(X,a)$.
\end{remark}

\begin{corollary}
  The identity map on $\V$ is continuous of type $\ftL\V\to(\V,{\xi_\le})$.
\end{corollary}

Recall from \cite[Proposition~3.29]{HR18} that the identity map on $X\times X$
is continuous of type
\begin{equation*}
  L_\V(X,a)\times L_\V(X,a) \longrightarrow L_\V((X,a)\otimes (X,a)),
\end{equation*}
for every $\V$-category $(X,a)$; hence, the composite map
\begin{equation*}
  L_\V(X,a)^\op\times L_\V(X,a) \longrightarrow L_\V((X,a)^\op\otimes (X,a))
  \xrightarrow{\quad a\quad}L_\V\V \longrightarrow (\V,{\xi_\le})
\end{equation*}
is continuous. Therefore, if $(X,a)$ is separated and $L_\V(X,a)$ is compact,
then these two structures define a $\V$-categorical compact Hausdorff space. In
fact, with $\Cats{\V}_\ch$ denoting the full subcategory of $\Cats{\V}_\sep$
defined by those $\V$-categories $X$ where $\ftL_\V$ is compact, the
construction above defines a functor $\Cats{\V}_\ch\to\CatCHs{\V}$.

For classical compact metric spaces, it is well-known that the Hausdorff
metric induces the hit-and-miss topology. Below we give an asymmetric version of
this result in the context of $\V$-categories.

\begin{lemma}\label{d:lem:10}
  For the $\V$-categorical compact Hausdorff space induced by a compact
  separated $\V$-category $X$, the hit-and-miss topology on $\ftH X$ coincides
  with the topology induced by the Hausdorff structure on $\ftH X$.
\end{lemma}
\begin{proof}
 Let $(X,a)$ be a compact separated $\V$-category. We show that the topology
 induced by $\ftH a$ is contained in the hit-and-miss topology; then, since the former is
 Hausdorff and the latter is compact, both topologies coincide.

 Let $A\in\ftH X$ and $u\lll k$. For every $v\in\V$ with $u\lll v\lll k$, put
 \begin{equation*}
   U_v=\bigcup_{x\in A}\OLB(x,v).
 \end{equation*}
 We show that $\OLB(A,u)=\bigcup_{u\lll v\lll k}U_v^\Box$. To see this, let
 $B\in\OLB(A,u)$, hence, $u\lll \ftH a(A,B)$. Let $v\in\V$ with $u\lll v\lll
 \ftH a(A,B)$. Then, for every $y\in B$, exists $x\in A$ with $v\lll a(x,y)$, that
 is, $y\in\OLB(x,v)$. Therefore $B\subseteq U_v$, which is equivalent to $B\in
 U_V^\Box$. Let now $B\in U_v^\Box$, for some $u\lll v\lll k$. Then, for all $y\in
 B$, there is some $x\in A$ with $v\lll a(x,y)$; hence
 \begin{equation*}
   u\lll v\le \bigwedge_{y\in B}\bigvee_{x\in A}a(x,y)=\ftH a(A,B).
 \end{equation*}
 Let now $B\in\ORB(A,u)$, and take $u',v\in\V$ with $u\lll u'\lll v\lll
 \ftH a(B,A)$. For every $x\in A$, there exists $y\in B$ with $v\lll a(y,x)$, that
 is, $y\in B\cap \ORB(x,v)$. Take $w\lll k$ with $u'\lll v\otimes w$. By
 compactness, there exist $x_1,\dots,x_n\in A$ with
 \begin{equation*}
   A\subseteq \OLB(x_1,w)\cup\dots\cup \OLB(x_n,w).
 \end{equation*}
 Then $B\in \ORB(x_1,v)^\Diamond\cap\dots\cap\ORB(x_n,v)^\Diamond$; moreover,
 $\ORB(x_1,v)^\Diamond\cap\dots\cap\ORB(x_n,v)^\Diamond\subseteq \ORB(A,u)$. To
 see the latter, let
 $B'\in \ORB(x_1,v)^\Diamond\cap\dots\cap\ORB(x_n,v)^\Diamond$ and $x\in A$,
 then $x\in\OLB(x_i,w)$ for some $i\in\{1,\dots,n\}$. Let
 $y\in B'\cap\ORB(x_i,v)$, then
 \begin{equation*}
   u'\lll v\otimes w\le a(y,x_i)\otimes a(x_i,x)\le a(y,x),
 \end{equation*}
 which implies $u\lll u'\le \ftH a(B',A)$.
\end{proof}

\begin{theorem}
  The functor $\ftH \colon\Cats{\V}\to\Cats{\V}$ restricts to the category
  $\Cats{\V}_\ch$, moreover, the diagram
  \begin{center}
    \begin{tikzcd}[row sep=large, column sep=large,ampersand replacement=\&]
      \Cats{\V}_\ch \ar[r,"\ftH"]\ar[d,""'] \& \Cats{\V}_\ch \ar[d,""] \\
      \CatCHs{\V} \ar[r,"\ftH"']            \& \CatCHs{\V}
    \end{tikzcd}
  \end{center}
  commutes.
\end{theorem}

\subsection{Coalgebras of Hausdorff polynomial functors on $\CatCHs{\V}$}

In this section we show that by ``adding topology'' we can improve the results
of Section~\ref{sec:coalg-hpf-V-Cat} about limits in categories of coalgebras
of Hausdorff polynomial functors. Throughout this section we still require
Assumptions~\ref{d:ass:4} and \ref{d:ass:1}.

We begin by showing that the category of coalgebras of the Hausdorff
functor on $\CatCHs{\V}$ is complete. The following result summarizes our
strategy.

\begin{theorem}
  \label{p:27}
  Let $\catX$ be a category that is complete, cocomplete and has an $(E,
  M)$-factorisation structure such that $\catX$ is $M$-wellpowered and $E$ is
  contained in the class of $\catX$-epimorphisms.  If a functor $\ftF \colon
  \catX \to \catX$ sends morphisms in $M$ to morphisms in $M$ and preserves
  codirected limits, then the category of coalgebras of $\ftF$ is complete.
\end{theorem}
\begin{proof}
  The claim follows by combining Corollary~\ref{p:19}, \cite[Proposition 7 of
    Section 9.4]{BW85}, \cite[Remark~4.4]{Ada05} and \cite[Corollary~2]{Lin69}.
\end{proof}

Also, the theorem bellow will help us to replace  ``preserves codirected
limits'' with ``preserves codirected initial cones''.

\begin{theorem}[{\cite[Proposition 13.15]{AHS90}}]
  \label{p:31}
  Let $\ftII{-} \colon \catX \to \catA$ be a limit preserving faithful functor and
  $\ftD \colon \catI \to \catX$ a diagram. A cone $\mathcal{C}$ for $\ftD$ is a
  limit in $\catX$ if and only if the cone $\ftII{\mathcal{C}}$ is a limit of
  $\ftII{\ftD}$ in $\catA$ and $\mathcal{C}$ is initial with respect to
  $\ftII{-}$.
\end{theorem}

\begin{proposition}
  \label{p:33}
  The Hausdorff functor on $\CatCHs{\V}$ preserves codirected initial cones with
  respect to the forgetful functor $\CatCHs{\V} \to \COMPHAUS$.
\end{proposition}
\begin{proof}
  Let $(f_i \colon (X,a,\alpha) \to (X_i,a_i,\alpha_i))_{i \in \catI}$ be a codirected initial cone with
  respect to the functor $\CatCHs{\V} \to \COMPHAUS$.
  We will show that for every $A,B \subseteq X$ the inequality
  \[
    u = \bigwedge_{i \in \catI} \ftH a_i(\ftH f_i(A), \ftH f_i(B)) \leq \ftH a(A,B)
  \]
  holds. Note that since $\V$ is (ccd) it is sufficient to prove that $v \leq
  \ftH a(A,B)$ for every $v \lll u$.

  Let $b \in B$ and fix $v \in \V$ such that $v \lll u$. Then, for every $i \in
  \catI$,
  \[
    u \leq \ftH a_i (f_i(A), f_i(B)) \leq \bigvee_{x \in A} a_i(f_i(x), f_i(b)),
  \]
  since $\ftH a_i(\ftH f_i(A), \ftH f_i(B)) = \ftH a_i (f_i(A), f_i(B))$ by
  lemma~\ref{d:lem:6}. Hence, for every $i \in \catI$, there exists an element
  $x_i \in A$ such that $v \leq a_i(f_i(x_i), f_i(b))$. Thus, for every $i \in
  \catI$, the set
  \[
    A_i = A \cap \{ x \in X \mid v \leq a_i(f_i(x), f_i(b)) \}
  \]
  is non-empty and closed because $\upc v \subseteq \V$ is closed
  (see Remark~\ref{d:rem:6}) and $a \colon (X,\alpha) \to (\V,
  \xi_\leq)$ is continuous (see Proposition~\ref{d:prop:4}). This way we obtain
  a codirected family of closed subsets of $X$ that has the finite intersection
  property since the cone $(f_i)_{i \in \catI}$ is codirected.
  Consequently, by compactness of $X$, there exists $x_b \in \bigcap_{i \in
  \catI} A_i$ such that for every $i \in \catI$, $v \leq a_i(f_i(x_b), f_i(b))$.
  Therefore, $v \leq a(x_b, b)$ since the cone $(f_i)_{i \in \catI}$ is initial,
  which implies $v \leq \ftH a(A,B)$.
\end{proof}

\begin{corollary}
  \label{p:26}
  The Hausdorff functor on $\CatCHs{\V}$ preserves initial monomorphisms with
  respect to the forgetful functor $\CatCHs{\V} \to \COMPHAUS$.
\end{corollary}

\begin{theorem}\label{d:thm:5}
  The functor $\ftH \colon\CatCHs{\V}\to\CatCHs{\V}$ preserves codirected
  limits.
\end{theorem}
\begin{proof}
  From Proposition~\ref{p:6}, the diagram below commutes.
  \begin{center}
    \begin{tikzcd}[row sep=large, column sep=large,ampersand replacement=\&]
      \CatCHs{\V} \ar[r,"\ftH"]\ar[d,""'] \& \CatCHs{\V} \ar[d,""] \ar[r,""] \& \COMPHAUS \\
      \ORDCH \ar[r,"\ftH"'] \& \ORDCH \ar[ur,""]
    \end{tikzcd}
  \end{center}
  Therefore, taking into account Theorem~\ref{p:31}, the claim follows
  from Proposition~\ref{p:33}.
\end{proof}

\begin{corollary}
  For $\ftH \colon\CatCHs{\V} \to\CatCHs{\V}$, the forgetful functor
  $\CoAlg(\ftH)\to\CatCHs{\V}$ is comonadic.
\end{corollary}

\begin{corollary}
  The category of coalgebras of the Hausdorff functor $\ftH \colon \CatCHs{\V}
    \to \CatCHs{\V}$ is complete. Moreover, the functor $\CoAlg{\ftH} \to
    \CatCHs{\V}$ preserves codirected limits.
\end{corollary}
\begin{proof}
  Being a topological category over $\COMPHAUS$, the category $\CatCHs{\V}$ is
  (surjective, initial mono)-structured. Therefore, the category $\CatCHs{\V}$
  satisfies all conditions necessary to apply Theorem~\ref{p:27}. Furthermore,
  the previous results show that $\ftH$ also satisfies the necessary requirements
  to apply Theorem~\ref{p:27}.
\end{proof}

In the sequel we describe the terminal coalgebra of the Hausdorff functor on
$\CatCHs{\V}$; which is the limit of the codirected diagram
\begin{equation}
  \label{d:eq:2}
  1\longleftarrow \ftH 1\longleftarrow \ftH \ftH 1\longleftarrow \dots,
\end{equation}
where the morphisms are obtained by applying successively $\ftH$ to the unique
morphism $f_! \colon \ftH 1 \to 1$.

First, we analyse the case of $\V=2$. To do so, let $(X,\tau^d)$
denote the discrete space with underlying set $X$, and observe that for every
positive integer $n$,

\begin{equation*}
%  \label{p:14}
  \ftH(n, \geq, \tau^d) = (n+1, \geq, \tau^d)
  \qquad\text{and}\qquad
  \ftH^nf_!(k)          = \min(k,n).
\end{equation*}

\begin{lemma}
  \label{p:15}
  Consider the one-point compactification $(\N+\infty, \tau^*)$ of the space
  $(\N,\tau^d)$.  The cone
  \begin{equation}
    \label{d:eq:3}
    (\min(-,n) \colon (\N+\infty,\geq, \tau^*) \longrightarrow (n+1, \geq,
    \tau^d))_{n \in \N}
  \end{equation}
  is a limit in $\ORDCH$ of the diagram \eqref{d:eq:2}.
\end{lemma}
\begin{proof}
  The assertion follows immediately from the ``Bourbaki'' criterion described in
  \cite[Theorem~3.29]{HNN19}: firstly, for every $n\in\N$, the map
  $\min(-,n) \colon (\N+\infty,\geq, \tau^*) \to (n+1, \geq, \tau^d)$ is
  surjective, monotone and continuous; secondly, the cone \eqref{d:eq:3} is
  point-separating and initial with respect to the canonical forgetful functor
  $\ORDCH\to\COMPHAUS$.
\end{proof}

\begin{theorem}
  \label{p:16}
  The map $f \colon (\N+\infty, \geq, \tau^*) \to \ftH(\N+\infty, \geq, \tau^*)$
  defined by
  \begin{equation*}
  f(n) = \begin{cases}
           \varnothing, & n = 0      \\
           \N+\infty,   & n = \infty \\
           \upc (n-1),  & \text{otherwise,}
         \end{cases}
  \end{equation*}
  is a terminal coalgebra for $\ftH \colon \ORDCH \to \ORDCH$.
\end{theorem}
\begin{proof}
  Since $\ftH \colon \ORDCH \to \ORDCH$ preserves codirected limits we can
  compute its terminal coalgebra from the limit of the diagram of
  Lemma~\ref{p:15}.  Therefore, the assertion holds by routine calculation.
\end{proof}

\begin{remark}
  The set $\N$ is an upset in $(\N+\infty, \geq, \tau^*)$ but it is not compact.
\end{remark}

As a consequence of the theorem above we can describe the terminal coalgebra of
the lower Vietoris functor on $\TOP$.

\begin{corollary}
  Consider the lower Vietoris functor $\ftV \colon \TOP \to \TOP$ and the space
  $(\N+\infty, \tau)$ whose topology is generated by the sets $[n, \infty]$, for
  $n \in \N$. The map $f \colon (\N+\infty, \tau) \to \ftV(\N+\infty,\tau)$
  defined by
  \begin{equation*}
  f(n) = \begin{cases}
           \varnothing, & n = 0      \\
           \N+\infty,   & n = \infty \\
           \upc (n-1),  & \text{otherwise.}
         \end{cases}
  \end{equation*}
  is a terminal coalgebra for $\ftV \colon \TOP \to \TOP$.
\end{corollary}
\begin{proof}
  The lower Vietoris functor $\ftV \colon \TOP \to \TOP$ restricts to the
  category $\STCOMP$ of stably compact spaces and spectral maps
  (see~\cite{Sch93}) which is isomorphic to the category $\ORDCH_\sep$
  (see~\cite{GHK+80}).  As observed in~\cite[Theorem~3.36]{HNN19}, the terminal
  coalgebra of the lower Vietoris functor on $\TOP$ can be obtained from the
  terminal coalgebra of the lower Vietoris on $\STCOMP$.  Since
  $\ftH \colon \ORDCH_\sep \to \ORDCH_\sep$ preserves codirected limits (see
  \cite[Corollary~3.33]{HNN19} or Theorem~\ref{d:thm:5} and
  Proposition~\ref{p:24}) and the limit and diagram of Lemma~\ref{p:15} actually
  live in $\ORDCH_\sep$, the claim follows by applying the functor
  \begin{equation*}
    % \ORDCH
    % \xrightarrow{\quad S\quad}
    \ORDCH_\sep
    \xrightarrow{\quad\simeq\quad} \STCOMP
    \longrightarrow \TOP
  \end{equation*}
  to the map of Theorem~\ref{p:16}.
\end{proof}

In the following we will see that the terminal coalgebra of
$\ftH \colon \CatCHs{\V} \to \CatCHs{\V}$ ``coincides'' with the terminal coalgebra
of $\ftH \colon \ORDCH \to \ORDCH$.

\begin{proposition}
  \label{p:23}
  Consider the lattice homomorphism $i \colon 2 \to \V$; that is $i(0) = \bot$
  and $i(1) = \top$.  The map $i$ induces a limit-preserving functor
  $\ftI \colon \ORDCH \to \CatCHs{\V}$ that keeps morphisms unchanged and sends
  an object $(X,a,\tau)$ of $\Cats{2}$ to $(X, i \cdot a, \tau)$.
\end{proposition}
\begin{proof}
  Let $(X, a, \tau)$ be an object of $\CatCHs{\V}$.  First, observe that $i$ is
  a lax homomorphism of quantales, hence $(X, a)$ is a $\V$-category;
  furthermore, it is clear that $i$ is a continuous function from
  $(2, \xi_\leq) \to (\V, \xi_\leq)$, hence by Proposition~\ref{d:prop:4},
  $(X, i \cdot a, \tau)$ defines an object of $\CatCHs{\V}$.  Now, a limit in
  $\CatCHs{\V}$ is a limit in $\COMPHAUS$ equipped with the initial structure
  with respect to the functor $\Cats{\V} \to \SET$.  Therefore, since $i$
  preserves infima, it follows that $\ftI \colon \ORDCH \to \CatCHs{\V}$
  preserves limits.
\end{proof}

\begin{corollary}
  The map
  $f \colon (\N+\infty, i \cdot \geq, \tau^*) \to \ftH(\N+\infty, i \cdot \geq, \tau^*)$
  defined by
  \begin{equation*}
    f(n) = \begin{cases}
      \varnothing, & n = 0      \\
      \N+\infty,   & n = \infty \\
      \upc (n-1),  & \text{otherwise,}
    \end{cases}
  \end{equation*}
  is a terminal coalgebra for $\ftH \colon \CatCHs{\V} \to \CatCHs{\V}$.
\end{corollary}
\begin{proof}
  Let $\ftH'$ denote the Hausdorff functor on $\CatCHs{2}$.  Since
  $\ftI \colon \CatCHs{2} \to \CatCHs{\V}$ preserves limits then $\ftI(1)$ is
  the terminal object in $\CatCHs{\V}$.  Moreover, the lattice homomorphism
  $\ftI \colon 2 \to \V$ preserves infima and suprema, thus we obtain
  $\ftI \cdot \ftH' = \ftH \cdot \ftI$.  Consequently,
  \[
    \ftI(1\longleftarrow \ftH'1\longleftarrow \ftH'\ftH'1\longleftarrow \dots) =
    1\longleftarrow \ftH1\longleftarrow \ftH\ftH1\longleftarrow \dots.
  \]
  Therefore, the claim follows from Theorem~\ref{p:16} and
  Proposition~\ref{p:23}.
\end{proof}

The corollary above affirms implicitly that, in general, the terminal coalgebra
of the Hausdorff functor on $\CatCHs{\V}$ is rather simple. After all,
independently of the quantale $\V$, we end up with a terminal coalgebra whose
carrier is an ordered set. Hausdorff polynomial functors seem far more
interesting in this regard.

\begin{definition}
  Let $\catX$ be a subcategory of $\CatCHs{\V}$ closed under finite limits and
  colimits such that the Hausdorff functor
  $\ftH \colon \CatCHs{\V} \to \CatCHs{\V}$ restricts to $\catX$.  We call a
  functor \df{Hausdorff polynomial} on $\catX$ if it belongs to the smallest
  class of endofunctors on $\catX$ that contains the identity functor, all
  constant functors and is closed under composition with $\ftH$, products and
  sums of functors.
\end{definition}

\begin{proposition}
  Every Hausdorff polynomial functor on $\CatCHs{\V}$ preserves initial
  monomorphisms with respect to the functor $\CatCHs{\V} \to \COMPHAUS$.
\end{proposition}
\begin{proof}
  Immediate consequence of Corollary~\ref{p:26} since the remaining cases
  trivially preserve initial monomorphisms.
\end{proof}

\begin{proposition}
  Every Hausdorff polynomial functor on $\CatCHs{\V}$ preserves codirected
  limits.
\end{proposition}
\begin{proof}
  We already know from Theorem~\ref{d:thm:5} that
  $\ftH \colon \CatCHs{\V} \to \CatCHs{\V}$ preserves codirected
  limits. Moreover, a routine calculation reveals that the sum of functors that
  preserve codirected initial cones with respect to the forgetful functor
  $\CatCHs{\V} \to \COMPHAUS$ also does so.  Consequently, the sum preserves
  codirected limits by Theorem~\ref{p:31} since the sum on $\COMPHAUS$
  preserves codirected limits (for instance, see~\cite{HNN19}).  The remaining
  cases are trivial.
\end{proof}

In light of the previous results, now we can apply Theorem~\ref{p:27} to obtain:

\begin{theorem}
  \label{d:thm:6}
  The category of coalgebras of a Hausdorff polynomial functor on $\CatCHs{\V}$
  is (co)complete.
\end{theorem}

Note that for Hausdorff polynomial functors, in general, we cannot apply the
same reasoning that led us to conclude that the terminal coalgebra of the
Hausdorff functor on $\CatCHs{\V}$ ``coincides'' with the terminal coalgebra of
the Hausdorff functor on $\ORDCH$.  For example, if $A$ ia a $\V$-categorical
compact Hausdorff space that does not come from an ordered set, then
applying the Hausdorff polynomial functor $\ftH
  \cdot (A \times \ftId)$ to the terminal object of $\CatCHs{\V}$ does not
necessarily yields a $\V$-category strucutre that comes from an
ordered set.

Now, by taking advantage of the results of
Appendix~\ref{sec:quant-enrich-categ}, we can deduce similar results for
Hausdorff polynomial functors on $\CatCHs{\V}_\sep$. However, to avoid
repetion, we conclude this paper by generalising the more interesting case of
Hausdorff polynomial functors on $\PRIEST$ discussed in \cite{HNN19}.

\begin{assumption}
  Until the end of the section we assume that $\V$ is a commutative and unital
  quantale such that for every $u \in \V$ the map $\hom(u,-) \colon (\V, \xi) \to
    (\V, \xi)$ is continuous.
\end{assumption}

\begin{definition}
  We call a $\V$-categorical compact Hausdorff space $X$ \df{Priestley} if the
  cone $\CatCHs{\V}(X,\V^\op)$ is initial and point-separating. We denote the
  full subcategory of $\CatCHs{\V}$ defined by all Priestley spaces by
  $\Priests{\V}$.
\end{definition}

\begin{example}
  For $\V=\two$, our notion of Priestley space coincides with the usual
  nomenclature for ordered compact Hausdorff spaces (see \cite{Pri70,Pri72}).
\end{example}

\begin{proposition}
  \label{p:35}
  The category $\Priests{\V}$ is closed under finite coproducts in
  $\CatCHs{\V}$.
\end{proposition}
\begin{proof}
  Let $A$ and $B$ be Priestley spaces. Note that for every morphism
  $f \colon A \to \V^\op$ and $g \colon B \to \V^\op$ in $\CatCHs{\V}$, the maps
  $f + \bot$ and $ \bot + g$, where $\bot$ represents the constant function
  $\bot$, are morphisms of type $A + B \to \V^\op$ in $\CatCHs{\V}$. Since $A$
  and $B$ are Priestley spaces, it follows that the cone of all these morphisms
  is initial and point-separating with respect to the functor
  $\CatCHs{\V} \to \COMPHAUS$.
\end{proof}

\begin{remark}
  \label{p:34}
  The inclusion functor $\Priests{\V}\hookrightarrow\CatCHs{\V}$ is right
  adjoint (see \cite[Theorem~16.8]{AHS90}); in particular $\Priests{\V}$ is
  complete and cocomplete and $\Priests{\V}\hookrightarrow\CatCHs{\V}$ preserves
  and reflects limits. Moreover, a mono-cone $(f_i \colon X\to X_i)_{i\in I}$ in
  $\Priests{\V}$ is initial with respect to $\Priests{\V}\to\COMPHAUS$ if and
  only if it is initial with respect to $\CatCHs{\V}\to\COMPHAUS$.
\end{remark}

\begin{proposition}
  \label{d:ex:1}
  The $\V$-categorical compact Hausdorff space $\V^\op$ is an algebra for $\mH$
  with algebra structure $\inf \colon\ftH \V^\op\to \V^\op$ that sends an
  element $A \in \ftH \V^\op$ to $\bigvee A$ (taken in $\V$).
\end{proposition}
\begin{proof}
  Clearly, $\inf \colon\ftH\V^\op\to\V^\op$ is a $\V$-functor. To see that
  $\inf$ is also continuous, recall from Remark~\ref{d:rem:6} that a subbasis
  for the Lawson topology of $\V$ is given by the sets
  \begin{equation*}
    (\downc v)^\complement=\{u\in\V\mid u\nleq v\}
    \quad\text{and}\quad
    (\upc v)^\complement=\{u\in\V\mid v\nleq u\}\qquad (v\in\V).
  \end{equation*}
  Note that $(\downc v)^\complement$ is decreasing in $\V^\op$ and $v^-$ is
  increasing in $\V^\op$. Let $v\in\V$. Then, for every $A\in\ftH(\V^\op)$,
  \begin{equation*}
    A\in{\inf}^{-1}((\downc v)^\complement)\iff \bigvee A\nleq v
    \iff A\cap (\downc v)^\complement\neq\varnothing;
  \end{equation*}
  that is,
  ${\inf}^{-1}((\downc v)^\complement)=((\downc v)^\complement)^\Diamond$. On
  the other hand,
  \begin{equation*}
    A\in{\inf}^{-1}((\upc v)^\complement)\iff
    v\nleq\bigvee A \iff
    \exists w\lll v\,\forall x\in A\,.\, w\nleq x \iff
    \exists w\lll v\,.\,A\subseteq (\upc w)^\complement;
  \end{equation*}
  that is,
  ${\inf}^{-1}((\upc v)^\complement)=\bigcup\{((\upc w)^\complement)^\Box\mid
  w\lll v\}$.
\end{proof}

In the sequel, given a morphism $\psi \colon X \to \V^\op$ of $\CatCHs{\V}$, we
denote by $\psi^\Diamond$ the composite

\begin{equation*}
  \ftH X \xrightarrow{\quad \ftH\psi \quad} \ftH(\V^\op) \xrightarrow{\quad
    \inf \quad} \V^\op
\end{equation*}

in $\CatCHs{\V}$.  With respect to the algebra structure of
Proposition~\ref{d:ex:1} above, we have:

\begin{proposition}
  \label{d:prop:3}
  Let $X$ be a $\V$-categorical compact Hausdorff space. Consider a
  $\V$-subcategory $\mathcal{R}\subseteq \V^X$ that is closed under finite
  weighted limits and such that $(\psi \colon X\to \V^\op)_{\psi\in\mathcal{R}}$
  is initial with respect to $\CatCHs{\V}\to\COMPHAUS$. Then the cone
  $(\psi^\Diamond \colon\ftH X\to \V^\op)_{\psi\in\mathcal{R}}$ is initial with
  respect to $\CatCHs{\V}\to\COMPHAUS$.
\end{proposition}
\begin{proof}
  We denote by $\ta$ the totally above relation of $\V$. Recall from
  Remark~\ref{d:rem:6} that, for every $u\in\V$, the set
  \begin{equation*}
    \tadown u=\{w\in\V\mid u\ta w\}
  \end{equation*}
  is open with respect to $\xi_\le$.  Let $A,B\in\ftH X$ and
  \begin{equation*}
    u\ta \ftH a(A,B)=\bigwedge_{y\in B}\bigvee_{x\in A}a(x,y).
  \end{equation*}
  Hence, there is some $y\in B$ so that, for all $x\in A$,
  \begin{equation*}
    u\ta a(x,y)=\bigwedge_{\psi\in\mathcal{R}}\hom(\psi(y),\psi(x)).
  \end{equation*}
  Let $x\in A$. There is some $\psi\in\mathcal{R}$ with
  $u\ta\hom(\psi(y),\psi(x))$; hence, with $v=\psi(y)$ and
  $\widehat{\psi}=\hom(v,\psi(-))\in\mathcal{R}$,
  \begin{equation*}
    u\ta\widehat{\psi}(x)
    \qquad\text{and}\qquad
    k\le \widehat{\psi}(y).
  \end{equation*}
  Therefore
  \begin{equation*}
    A\subseteq\bigcup\{\psi^{-1}(\tadown u)\mid \psi\in\mathcal{R},k\le\psi(y)\};
  \end{equation*}
  by compactness, there exist finitely many $\psi_1,\dots,\psi_n\in\mathcal{R}$
  so that $k\le\psi_i(y)$ and
  \begin{equation*}
    A\subseteq  \psi_1^{-1}(\tadown u)\cup \dots\cup \psi_n^{-1}(\tadown u).
  \end{equation*}
  Put $\widehat{\psi}=\psi_1\wedge\dots\wedge\psi_n\in\mathcal{R}$. Then
  $k\le \widehat{\psi}(y)$ and $u\ta \widehat{\psi}(x)$, for all $x\in
  A$. Therefore
  \begin{equation*}
    \hom(\widehat{\psi}^\Diamond(B),\widehat{\psi}^\Diamond(A))\le\hom(k,u)=u.\qedhere
  \end{equation*}
\end{proof}

\begin{proposition}
  \label{p:32}
  Let $(f \colon X\to X_i)_{i\in I}$ be a codirected cone in $\CatCHs{\V}$.
  Then
  \begin{equation*}
    \{ \varphi f_i \mid i \in I,\,\varphi \colon X_i \to \V^\op \in
    \CatCHs{\V} \} \subseteq \V^X
  \end{equation*}
  is closed under finite weighted limits.
\end{proposition}

By Proposition~\ref{d:prop:3}, the Hausdorff functor restricts to a functor
$\ftH\colon \Priests{\V}\to\Priests{\V}$, hence the Hausdorff monad $\mH$
restricts to $\Priests{\V}$.

\begin{theorem}%\label{d:thm:5}
  Every Hausdorff polynomial functor on $\Priests{\V}$ preserves codirected
  limits.
\end{theorem}
\begin{proof}
  Every Hausdorff polynomial functor on $\Priests{\V}$ corresponds to the
  restriction to $\Priests{\V}$ of a Hausdorff polynomial functor on
  $\CatCHs{\V}$ and the inclusion functor $\Priests{\V} \to \CatCHs{\V}$
  preserves and reflects limits (see Proposition~\ref{p:35} and
  Remark~\ref{p:34}).
\end{proof}

\begin{corollary}
  For every Hausdorff polynomial functor $\ftF$ on $\Priests{\V}$, the forgetful functor
  $\CoAlg(\ftF)\to\Priests{\V}$ is comonadic.
\end{corollary}

\begin{theorem}
  The category of coalgebras of a Hausdorff polynomial functor $\ftF$ on
  $\Priests{\V}$ is complete. Moreover, the functor
  $\CoAlg(\ftF) \to \Priests{\V}$ preserves codirected limits.
\end{theorem}
\begin{proof}
  The category $\Priests{\V}$ inherits the (surjective, initial
  mono-cone)-factorisation structure from $\CatCHs{\V}$.  Therefore, the
  previous discussion shows that we can apply Theorem~\ref{p:27}.
\end{proof}

\appendix
\section{Appendix}
\label{sec:quant-enrich-categ}

In this section we collect some facts about $\V$-categories and $\V$-functors,
for more information we refer to \cite{Law73,Stu14}. Furthermore, we present
some useful properties of the reflector into the category of separated
$\V$-categories that follow from standard arguments, but seem to be absent from
the literature.

\begin{definition}
  Let $\V$ be a commutative and unital quantale. A \df{$\V$-category} is a pair $(X,a)$ consisting of a
  set $X$ and a map $a \colon X\times X\to\V$ satisfying
  \begin{align*}
    k\le a(x,x) &&\text{and}&& a(x,y)\otimes a(y,z)\le a(x,z),
  \end{align*}
  for all $x,y,z\in X$. Given $\V$-categories $(X,a)$ and $(Y,b)$, a
  \df{$\V$-functor} $f \colon (X,a)\to (Y,b)$ is a map $f \colon X\to Y$ such
  that
  \[
    a(x,y) \le b(f(x),f(y)),
  \]
  for all $x,y \in X$.
\end{definition}

In particular, the quantale $\V$ becomes a $\V$-category with structure
$\hom \colon \V\times\V\to\V$.

For every $\V$-category $(X,a)$, $a^\circ(x,y)=a(y,x)$ defines another
$\V$-category structure on $X$, and the $\V$-category $(X,a)^\op:=(X,a^\circ)$
is called the \df{dual} of $(X,a)$.  A $\V$-category $(X,a)$ is called
\df{symmetric} whenever $(X,a)=(X,a)^\op$.

Clearly, $\V$-categories and $\V$-functors define a category, denoted as
$\Cats{\V}$. The full subcategory of $\Cats{\V}$ defined by all symmetric
$\V$-categories is denoted as $\Cats{\V}_\sym$.

\begin{remark}
  \label{p:12}
  Given $\V$-categories $(X,a)$ and $(Y,b)$, we define the tensor product of
  $(X,a)$ and $(Y,b)$ to be the $\V$-category
  $(X,a) \otimes (Y,b) = (X \times Y, a \otimes b)$, with
  \begin{equation*}
    a \otimes b ((x,y),(x',y')) = a(x,x') \otimes b(y,y').
  \end{equation*}
  This operation makes $\Cats{\V}$ a symmetric monoidal closed category, where
  the internal $\hom$ of $(X,a)$ and $(Y,b)$ is the $\V$-category
  $[(X,a), (Y,b)] = (\Cats{\V}((X,a),(X,b)), [-,-])$, with
  \begin{equation*}
    [f,g] = \bigwedge_{x \in X} b(f(x), g(x)).
  \end{equation*}
  We note that $[(X,a), (Y,b)]$ is a $\V$-subcategory of the $X$-fold product
  $(Y,b)^X$ of $(Y,b)$.
\end{remark}

The following propositions are particularly useful to construct $\V$-functors
when combined with the fact that $\Cats{\V}$ is symmetrical monoidal closed.

\begin{proposition}
  \label{p:9}
  For every set $I$, the assignments $f \mapsto \bigvee_{i \in I} f(i)$ and
  $f \mapsto \bigwedge_{i \in I} f(i)$ define $\V$-functors of type $\V^I\to\V$.
\end{proposition}

\begin{proposition}
  \label{p:10}
  For every $\V$-category $(X,a)$, the map
  $a \colon (X,a)^\op \otimes (X,a) \to (\V, \hom)$ is a $\V$-functor.
\end{proposition}

The category $\Cats{\V}$ is well behaved regarding (co)limits.

\begin{theorem}
  \label{p:22}
  The canonical forgetful functor $\Cats{\V}\to\SET$ is topological. For a
  structured cone $(f_i \colon X\to (X_i,a_i))$, the initial lift $(X,a)$ is
  given by
  \begin{equation*}
    a(x,y) = \bigwedge_{i\in I}a_i(f_i(x),f_i(y)),
  \end{equation*}
  for all $x,y\in X$. Moreover, $\Cats{\V}_\sym$ is closed in $\Cats{\V}$ under
  initial cones; therefore the canonical forgetful functor
  $\Cats{\V}_\sym\to\SET$ is topological as well, and the inclusion functor
  $\Cats{\V}_\sym\hookrightarrow\Cats{\V}$ has a left adjoint.
\end{theorem}

We also recall that $\Cats{\V}_\sym\hookrightarrow\Cats{\V}$ has a concrete
right adjoint which sends the $\V$-category $(X,a)$ to its \df{symmetrisation}
$(X,a_s)$ given by
\begin{equation*}
  a_s(x,y) = a(x,y) \wedge a(y,x),
\end{equation*}
for all $x,y\in X$.

Every $\V$-category $(X,a)$ carries a natural order defined by
\[
  x\le y \text{ whenever } k\le a(x,y),
\]
which can be extended pointwise to $\V$-functors making $\Cats{\V}$ a
\emph{2-category}. The natural order of $\V$-categories defines a faithful
functor $\Cats{\V}\to\ORD$. A $\V$-category is called \df{separated} whenever
its underlying ordered set is anti-symmetric, and we denote by $\Cats{\V}_\sep$
the full subcategory of $\Cats{\V}$ defined by all separated
$\V$-categories. Tautologically, an ordered set is separated if and only if it
is anti-symmetric.

\begin{theorem}
  $\Cats{\V}_\sep$ is closed in $\Cats{\V}$ under monocones. Hence, the
  forgetful functor $\Cats{\V}_\sep\to\SET$ is mono-topological and the
  inclusion functor $\Cats{\V}_\sep\hookrightarrow\Cats{\V}$ has a left adjoint.
\end{theorem}

Let us describe the left adjoint $S \colon \Cats{\V}\to \Cats{\V}_\sep$ of
$\Cats{\V}_\sep\hookrightarrow\Cats{\V}$. To do so, consider a $\V$-category
$(X,a)$. Then
\begin{equation*}
  x\sim y \qquad\text{whenever}\qquad x\le y\quad\text{and}\quad y\le x
\end{equation*}
defines an equivalence relation on $X$, and the quotient set $X/{\sim}$ becomes
a $\V$-category $(X/{\sim},\widetilde{a})$ by putting
\begin{equation}\label{d:eq:4}
  \widetilde{a}([x],[y])=a(x,y);
\end{equation}
this is indeed independent of the choice of representants of the equivalence
classes. Then the projection map
\begin{equation*}
  q_{(X,a)} \colon X \longrightarrow X/{\sim},\,x\longmapsto[x]
\end{equation*}
is a $\V$-functor $q_{(X,a)} \colon (X,a)\to(X/{\sim},\widetilde{a})$, it is
indeed the unit of this adjunction at $(X,a)$. Furthermore, by \eqref{d:eq:4},
$q_{(X,a)} \colon (X,a)\to(X/{\sim},\widetilde{a})$ is a universal quotient and
initial with respect to $\Cats{\V}\to\SET$.

\begin{lemma}
  \label{p:2}
  A cone $(f_i \colon (X,a)\to (X_i,a_i))_{i\in I}$ in $\Cats{\V}_\sep$ is
  initial with respect to $\Cats{\V}_\sep\to\SET$ if and only if
  \begin{equation}
    \label{p:5}
    a(x,y)=\bigwedge_{i\in I}a_i(f_i(x),f_i(y)),
  \end{equation}
  for all $x,y\in X$.
\end{lemma}
\begin{proof}
  Clearly, if \eqref{p:5} is satisfied then
  $(f_i \colon (X,a)\to (X_i,a_i))_{i\in I}$ is initial with respect to
  $\Cats{\V}_\sep \to \SET$ since it is initial with respect to
  $\Cats{\V} \to \SET$.  Suppose now that
  $(f_i \colon (X,a)\to (X_i,a_i))_{i\in I}$ is initial with respect to
  $\Cats{\V}_\sep \to \SET$.  Fix $x,y\in X$. Then
  \[
    a(x,y) \leq \bigwedge_{i\in I}a_i(f_i(x),f_i(y))=u
  \]
  because $f_i \colon (X,a) \to (X_i, a_i)$ is a $\V$-functor for every
  $i \in I$. It is left to show that $u\le a(x,y)$. This is certainly true if
  $u=\bot$; assume now that $\bot<u$. Let $2_u$ be the separated $\V$-category
  with underlying set $\{0,1\}$ and structure $a_u$ defined by
  \begin{equation*}
    a_u(0,1)=u,\quad a_u(0,0)=a_u(1,1)=k,\quad \text{and}\quad a_u(1,0)=\bot.
  \end{equation*}
  Consider $h \colon \{0,1\}\to X$ with $h(0)=x$ and $h(1)=y$. Then $f_i\cdot h$
  is a $\V$-functor, for every $i\in I$. Hence, since
  $(f_i \colon (X,a)\to (X_i,a_i))_{i\in I}$ is initial, $h \colon 2_u\to X$ is
  a $\V$-functor, which implies $u\le a(x,y)$.
\end{proof}

\begin{corollary}\label{d:cor:4}
  The functor $\ftS \colon \Cats{\V}\to \Cats{\V}_\sep$ preserves initial cones
  with respect to the canonical forgetful functors.
\end{corollary}
\begin{proof}
  Let $(f_i \colon (X,a)\to (X_i,a_i))_{i\in I}$ be an initial cone with respect
  to $\Cats{\V} \to \SET$.  Then, for every
  $[x],[y] \in \ftS (X,a)=(X/{\sim},\widetilde{a})$, and with
  $\ftS (X_i,a_i)=(X/{\sim},\widetilde{a}_i)$ for all $i\in I$,
  \[
    \widetilde{a}([x],[y]) = a(x,y)
                           = \bigwedge_{i\in I} a_i(f_i(x),f_i(y))
                           = \bigwedge_{i\in I} \widetilde{a_i}([f_i(x)],[f_i(y)])
                           = \bigwedge_{i\in I} \widetilde{a_i}(\ftS f_i([x]),\ftS f_i([y])).
  \]
  Therefore, the claim follows by Lemma~\ref{p:2}.
\end{proof}

\begin{remark}\label{d:rem:5}
  In \cite{CHR20} it is shown that $\ftS \colon \Cats{\V}\to \Cats{\V}_\sep$
  preserves finite products. However, $\ftS$ does not preserve limits in
  general, in particular, $\ftS$ does not preserve codirected limits. For
  instance, consider the ``empty limit'' of \cite{Wat72} and equip every $X_i$
  ($i\in I$) with the indiscrete $\V$-category structure $a_i$ where
  $a_i(x,y)=\top$ for all $x,y\in X_i$. Then $S(X_i,a_i)$ has exactly one
  element, for each $i\in I$; hence the limit of the corresponding diagram in
  $\Cats{\V}_\sep$ has one element.
\end{remark}

% \bibliographystyle{halpha}
% \bibliography{bibliography_dirk}
\newcommand{\etalchar}[1]{$^{#1}$}

\end{document}